\documentclass[12pt]{article}
\usepackage{graphicx,gastex}
\usepackage{amssymb,amsfonts,amsthm}
\usepackage{amsmath,amsthm,amssymb}
\usepackage{amscd}
\usepackage{amsfonts}
\usepackage{color,url}
\newcommand{\con}{\operatorname{con}}
\newcommand{\Id}{\operatorname{Id}}

\newcommand{\mul}{\operatorname{mul}}
\newcommand{\simp}{\operatorname{sim}}
\newcommand{\ini}{\operatorname{ini}}
\newcommand{\fin}{\operatorname{fin}}
\newcommand{\Dist}{\operatorname{Dist}}
\newcommand{\var}{\operatorname{var}}

\newtheorem{theorem}{Theorem}[section]
\newtheorem{ex}[theorem]{Example}
\newtheorem{cor}[theorem]{Corollary}
\newtheorem{sufcon}[theorem]{Sufficient Condition}
\newtheorem{fact}[theorem]{Fact}
\newtheorem{lemma}[theorem]{Lemma}

\newtheorem{prop}[theorem]{Proposition}
\newtheorem{obs}[theorem]{Observation}

\newtheorem{sortlemma}{Sorting Lemma}

\theoremstyle{definition}
\newtheorem{question}{Question}

\begin{document}

\date{}

\title{Limit varieties of aperiodic monoids}
\author{Sergey V. Gusev and Olga  B. Sapir}

\maketitle

\renewcommand{\thefootnote}{\fnsymbol{footnote}}
\footnotetext{2010 \textit{Mathematics subject classification}. 20M07.}   
\footnotetext{\textit{Key words and phrases}. Monoid, variety, limit variety, finite basis problem.} 
\footnotetext{S. V. Gusev was supported by grant No.~25-71-00005 from the Russian Science Foundation, https://rscf.ru/en/project/25-71-00005/.}   
\renewcommand{\thefootnote}{\arabic{footnote}} 

\begin{abstract}
A {\it limit variety} is a variety that is minimal with respect to being non-finitely based.   
We find the fourteenth example of a limit variety of aperiodic monoids.
We also show that if there exists any other limit variety of aperiodic monoids, then it is contained in the join of the  variety generated by the \mbox{3-}element cyclic monoid and the variety of all idempotent monoids.
\end{abstract}

\section{Introduction}
\label{sec: intr}

A variety of algebras is called \textit{finitely based} (abbreviated to FB) if it has a finite basis of its identities, otherwise, the variety is said to be \textit{non-finitely based} (abbreviated to NFB).
A variety is \textit{hereditary finitely based} (abbreviated to HFB) if all its subvarieties are FB.
A variety is called a \textit{limit variety} if it is NFB but all its proper subvarieties are FB.

According to Zorn's lemma, a variety is HFB if and only if it excludes all limit varieties.
This explains the interest in limit varieties.
However, constructing explicit examples of limit varieties turns out to be extremely non-trivial.
For instance, there are uncountably many limit varieties of periodic groups~\cite{Kozhevnikov-12}, while no explicit example of a limit variety of groups is known.
Locating such an example remains one of the intriguing open problems in the theory of group varieties; see Section~3 in the survey article~\cite{Gupta-Krasilnikov-03}.

The present article is concerned with the limit varieties of \textit{monoids}, i.e., semigroups with an identity element.
A complete classification of all limit varieties of monoids seems to be highly infeasible since that would include a description of all limit varieties of periodic groups.
Therefore, it is logical to focus on the class of monoids with only trivial subgroups. 
Such monoids are called \textit{aperiodic} monoids.
The first two examples of limit varieties of aperiodic monoids were presented by Jackson~\cite{Jackson-05} in 2005.
Since then, limit varieties of aperiodic monoids have received much attention and several more examples have been found as well as some partial descriptions have been obtained; see the articles~\cite{Gusev-20,Gusev-21,Gusev-Sapir-22,Lee-09,Lee-12,Sapir-21,Sapir-23,Zhang-13,Zhang-Luo-19} and Section~11 of the recent monograph~\cite{Lee-23}.
We will discuss these results in detail in the next section.

In this paper, based on the previous results, we present a new example of a limit variety of aperiodic monoids.
We also show that if there is any other limit variety of aperiodic monoids, then it is contained in $\mathbb Q^1 \vee \mathbb B^{1}$, where $\mathbb B^{1}$ is the variety of all idempotent monoids and $\mathbb Q^1$ is the variety generated by the following six-element $J$-trivial monoid
\[
Q^1 = \langle e, b, c \mid  e^2 = e, eb = b, ce = c, ec = be = cb = 0 \rangle\cup\{1\}.
\]

The article is structured as follows.
In Section~\ref{sec: known}, we provide a brief overview of the previous results on limit varieties of aperiodic monoids.
Several basic facts are collected in Section~\ref{sec: prelim}.
In Section~\ref{sec: idempotent}, we remind a description of the lattice of varieties of idempotent monoids and some results on identities of idempotent monoids.
Based on the description of the subvariety lattice of $\mathbb B^{1}$, we clarify the structure of the subvariety lattice of the variety $\mathbb Q^1\vee\mathbb B^{1}$ in Section~\ref{sec: QB}. 
Section~\ref{sec: EEhfb} provides a new example of a HFB variety of monoids.
Section~\ref{sec: new lim} presents a new explicit example of a limit variety of aperiodic monoids.
Finally, in Section~\ref{sec: new sort}, we reduce the further study of limit varieties of aperiodic monoids to subvarieties of $\mathbb Q^1 \vee \mathbb B^1$.

\section{A brief overview of the previous results on limit varieties}
\label{sec: known}

We need some definitions and notation.
Let $\mathfrak A$ be a countably infinite set called an \textit{alphabet}. 
As usual, let~$\mathfrak A^+$ and~$\mathfrak A^\ast$ denote the free semigroup and the free monoid over the alphabet~$\mathfrak A$, respectively. 
Elements of~$\mathfrak A$ are called \textit{letters} and elements of~$\mathfrak A^\ast$ are called \textit{words}.
For any set $W$ of words, let $M(W)$ denote the Rees quotient monoid of $\mathfrak A^\ast$ over the ideal consisting of all words that are not subwords of any word in $W$. 
We denote by $\mathbb M(W)$ the variety generated by $M(W)$.

The varieties $\mathbb M(\{xzxyty\})$ and $\mathbb M(\{xyzxty,xzytxy\})$ are the two  above-men\-tioned first examples of limit varieties constructed by Jackson~\cite{Jackson-05}.
In~\cite{Lee-09}, Lee proved the uniqueness of the limit varieties $\mathbb M(\{xzxyty\})$ and $\mathbb M(\{xyzxty,xzytxy\})$ in the class of varieties of finitely generated aperiodic monoids with central idempotents.
In~\cite{Lee-12}, Lee generalized the result of~\cite{Lee-09} and established that $\mathbb M(\{xzxyty\})$ and $\mathbb M(\{xyzxty,xzytxy\})$ are the only limit varieties within the class of varieties of aperiodic monoids with central idempotents.
In 2013, Zhang~\cite{Zhang-13} found a NBF variety of monoids that contains neither $\mathbb M(\{xzxyty\})$ nor $\mathbb M(\{xyzxty,xzytxy\})$ and, therefore, she proved that there exists a limit variety of monoids that differs from $\mathbb M(\{xzxyty\})$ and $\mathbb M(\{xyzxty,xzytxy\})$. 

If $S$ is a semigroup, then the monoid obtained by adjoining a new identity element to $S$ is denoted by $S^1$.
If $M$ is a monoid, then the variety of monoids generated by $M$ is denoted by $\mathbb M$.
If $\mathbb V$ is a monoid variety, then  $\overline{\mathbb V}$ denotes the variety \textit{dual to} $\mathbb V$, i.e., the variety consisting of monoids anti-isomorphic to monoids from $\mathbb V$.

In~\cite{Zhang-Luo-19}, Zhang and Luo found the third explicit example of a limit variety of aperiodic monoids.
It is the variety $\mathbb A^1 \vee\overline{\mathbb A^1}$, where
\[
A=\langle a,b,c\mid a^2=a,\,b^2=b,\,ab=ca=0,\,ac=cb=c\rangle.
\]
The semigroup $A$ was introduced and shown to generate a FB variety in \cite[Section~19]{Lee-Zhang-15}.

For any $n\ge 1$, let $S_n$ denote the full symmetric group on the set $\{1,2,\dots,n\}$.
Let $\var\Sigma$ denote the variety defined by a set $\Sigma$ of identities. 

The varieties
\[
\mathbb J=\var
\left\{\!\!\!
\begin{array}{l}
x^2y^2\approx y^2x^2, xyx\approx xyx^2, xyzxy\approx yxzxy, xyxztx\approx xyxzxtx,\\
xz_{1\pi}z_{2\pi}\cdots z_{n\pi}xt_1z_1\cdots t_nz_n\approx x^2z_{1\pi}z_{2\pi}\cdots z_{n\pi}t_1z_1\cdots t_nz_n
\end{array}
\middle\vert
\begin{array}{l}
n\ge1,\\
\pi\in S_n
\end{array}
\!\!\!\right\},
\]
and $\mathbb {\overline J}$ are the next pair of limit varieties of monoids~\cite{Gusev-20}.
In~\cite{Gusev-21}, it was proved that $\mathbb M(\{xzxyty\})$, $\mathbb M(\{xyzxty,xzytxy\})$,
$\mathbb J$ and $\mathbb {\overline J}$ are the only limit varieties of aperiodic monoids with commuting idempotents.
In~\cite{Gusev-Sapir-22}, the authors  presented one more pair $\mathbb K^1$ and $\overline {\mathbb K^1}$ of limit varieties,
where
\[
K =\langle a,b,c \mid a^2=a,\,b^2=b^3,\,bcb^2=bcb,\,ca=c,\,abc=ac=ba=b^2c=0\rangle,
\] 
and showed that there are exactly seven limit varieties of $J$-trivial monoids.

Lee and Li \cite[Section~14]{Lee-Li-11} considered the semigroup
\[
E = \langle a, b, c \mid a^2 = ab = 0, ba = ca = a, b^2 = bc = b, c^2= cb = c \rangle
\]
and showed that 
\[
\mathbb E^1=\var\{xtx \approx xtx^2 \approx x^2tx,\, xy^2x \approx x^2y^2\}.
\]
For each $n\ge1$, let
\[
\mathbb E^1 \{\sigma_n\} = \var \{xtx \approx xtx^2 \approx x^2tx,\, xy^2x \approx x^2y^2, \sigma_n\}.
\]
where
\[
\sigma_n: \mathbf e_1t_1\mathbf e_2t_2\cdots\mathbf e_nt_nx^2y^2\approx\mathbf e_1t_1\mathbf e_2t_2\cdots\mathbf e_nt_ny^2x^2,
\]
and $\mathbf e_{2i-1}=x$ and $\mathbf e_{2i}=y$ for all $i\ge1$.
According to Proposition 5.6 in~\cite{Jackson-Lee-18}, the lattice of subvarieties of $\mathbb E^1$ contains an infinite ascending chain $\mathbb E^1\{\sigma_1\} \subset \mathbb E^1\{\sigma_2\} \subset \cdots \subset \mathbb E^1\{\sigma_n\} \subset \cdots$.
In~\cite{Sapir-23}, the second-named author proved that the varieties $\mathbb A^1\vee \mathbb E^1 \{\sigma_2\}$ and $\overline{\mathbb A^1}\vee \overline{\mathbb E^1 \{\sigma_2\}}$ are also limit.

In the recent paper~\cite{Gusev-Li-Zhang-25}, the first-named author, Li and Zhang proved that the following four varieties
\[
\mathbb J_1=\var
\left\{\!\!\!
\begin{array}{l}
xyxz\approx xyxzx, xyxty\approx x^2yty,\\
(x_1t_1x_2t_2\cdots x_{n}t_{n}) \,x\, (z_{1}z_{2}\cdots z_{2n}) \,y\, (t_{n+1}x_{n+1}\cdots t_{2n}x_{2n})\cdot\\
txy\, (x_{1\pi}z_{1\tau}x_{2\pi}z_{2\tau}\cdots x_{(2n)\pi}z_{(2n)\tau}) \approx\\
(x_1t_1x_2t_2\cdots x_{n}t_{n}) \,x\, (z_{1}z_{2}\cdots z_{2n}) \,y\, (t_{n+1}x_{n+1}\cdots t_{2n}x_{2n})\cdot\\
tyx\, (x_{1\pi}z_{1\tau}x_{2\pi}z_{2\tau}\cdots x_{(2n)\pi}z_{(2n)\tau})
\end{array}
\middle\vert
\begin{array}{l}
n\ge1,\\
\pi,\tau\in S_{2n}
\end{array}
\!\!\!\right\},
\]
\[
\mathbb J_2=\var\!
\left\{\!\!\!
\begin{array}{l}
xyxz\approx xyxzx,x y_1 y s xy x_1y_1tx_1\approx x y_1 y s xy x_1y_1tx_1,\\
(x_1t_1x_2t_2\cdots x_{2n}t_{2n})\, x\, (z_1s_1z_2s_2\cdots z_ns_n)\, y\, (z_{n+1}z_{n+2}\cdots z_{2n})\cdot\\
xy\,(x_{1\pi}z_{1\tau}x_{2\pi}z_{2\tau}\cdots x_{(2n)\pi}z_{(2n)\tau})\,t\,(s_1s_2\cdots s_n) \approx\\
(x_1t_1x_2t_2\cdots x_{2n}t_{2n})\, x\, (z_1s_1z_2s_2\cdots z_ns_n)\, y\, (z_{n+1}z_{n+2}\cdots z_{2n})\cdot\\
yx\,(x_{1\pi}z_{1\tau}x_{2\pi}z_{2\tau}\cdots x_{(2n)\pi}z_{(2n)\tau})\,t\,(s_1s_2\cdots s_n)
\end{array}
\!\middle\vert\!
\begin{array}{l}
n\ge1,\\
\pi,\tau\in S_{2n}
\end{array}
\!\!\!\!\right\},
\]
$\overline{\mathbb J_1}$ and $\overline{\mathbb J_2}$ are also limit and no other limit variety satisfying $xyxz\approx xyxzx$ exists.

Let $A_0$ be the semigroup given by presentation:
\[
A_0 = \langle  a, b \mid a^2=a, b^2=b, ab = 0 \rangle.
\]
The monoid $A_0^1$ was introduced and shown to be FB in~\cite{Edmunds-77}.
In the present paper, we show that the variety $\mathbb A_0^1\vee \mathbb E^1 \{\sigma_2\}\vee \overline{\mathbb E^1 \{\sigma_2\}}$ is a new example of a limit variety mentioned in the introduction.

\section{Preliminaries}
\label{sec: prelim}

A letter is called \textit{simple} [\textit{multiple}] in a word $\bf u$ if it occurs in $\bf u$ once [at least
twice].   
The set of all  letters in $\bf u$ is denoted by $\con({\bf u})$. 
Notice that  $\con({\bf u}) = \simp({\bf u}) \cup \mul({\bf u})$, where $\simp({\bf u})$ is the set of all simple letters in $\bf u$ and $\mul({\bf u})$ is the set of all multiple letters in $\bf u$.

\begin{fact}[\mdseries{\cite[Proposition~4.3]{Lee-Li-11}}]
\label{F: Q}
 An identity ${\bf u} \approx  {\bf v}$ holds in $\mathbb Q^1$ if and only if ${\bf u}= \mathbf a_0 \prod_{i=1}^m (t_i\mathbf a_i)$ and  ${\bf v} = \mathbf b_0 \prod_{i=1}^m (t_i\mathbf b_i)$ for some $m \ge 0$, $\simp({\bf u})= \simp({\bf v})=\{t_1,t_2,\dots, t_m\}$, and $\con({\bf a}_i) = \con({\bf b}_i)$ for each $0 \le i \le m$.\qed
\end{fact}

 A \textit{block} of a word $\mathbf u$ is a maximal subword of $\mathbf u$ that does not contain any letters simple in $\mathbf u$.
If ${\bf u}= \mathbf a_0 \prod_{i=1}^m (t_i\mathbf a_i)$ and  ${\bf v} = \mathbf b_0 \prod_{i=1}^m (t_i\mathbf b_i)$ with $ \simp({\bf u})= \simp({\bf v})=\{t_1,t_2,\dots, t_m\}$, then the blocks $\mathbf a_i$ and $\mathbf b_i$ are called \textit{corresponding}.

Let 
\[
\mathbb L_2^1 = \var\{x \approx x^2, xy \approx xyx\}\ \text{ and }
\mathbb R_2^1 = \var\{x \approx x^2, xy \approx yxy\}.
\]
Let $\ini(\mathbf w)$ [respectively, $\fin(\mathbf w)$] denote the word obtained from $\mathbf w$ by retaining the first [respectively, last] occurrence of each letter.
The following description of identities of $\mathbb L_2^1$ and $\mathbb R_2^1$ is well-known and can be easily verified.

\begin{fact}
\label{F: L2}
\quad
\begin{itemize}
\item[\textup{(i)}] An identity ${\bf u} \approx  {\bf v}$ holds in $\mathbb L_2^1$ if and only if $\ini({\bf u}) = \ini({\bf v})$.
\item[\textup{(ii)}] An identity ${\bf u} \approx  {\bf v}$ holds in $\mathbb R_2^1$ if and only if $\fin({\bf u}) = \fin({\bf v})$.\qed
\end{itemize}
\end{fact}

\begin{fact}[\mdseries{\cite[Section~14.3]{Lee-Li-11}}]
\label{F: E}
An identity ${\bf u} \approx  {\bf v}$ holds in $\mathbb E^1$ if and only if ${\bf u} \approx {\bf v}$ holds in $\mathbb Q^1$ and the corresponding blocks of $\bf u$ and $\bf v$ are equivalent within $\mathbb L_2^1$.\qed
\end{fact}

\begin{lemma}[\mdseries{\cite[Lemma~4.3]{Sapir-23}}]
\label{L: not E} 
Let $\mathbb V$ be a monoid variety that satisfies  $xtx \approx x^2tx \approx xtx^2$ and contains neither  $\mathbb E^1\{\sigma_2\}$ nor $\overline{\mathbb A^1}$.
Then  $\mathbb V$ is FB.\qed
\end{lemma}

We say that a set of words $W \subseteq \mathfrak A^{\ast}$ is {\em stable with respect to a monoid variety $\mathbb V$} if  ${\bf v} \in W$ whenever ${\bf u} \in W$ and $\mathbb V$ satisfies ${\bf u} \approx {\bf v}$. Recall that a word ${\bf u} \in  \mathfrak A^{\ast}$ is an \textit{isoterm} \cite{Perkins-69} for $\mathbb V$ if the set $\{ {\bf u} \}$ is stable with respect to $\mathbb V$.

If $\tau$ is an equivalence relation on the free monoid $\mathfrak A^{\ast}$ and $\mathbb V$ is a monoid variety, then a word ${\bf u} \in \mathfrak A^{\ast}$ is said to be  a $\tau$-\textit{term} for $\mathbb V$ if ${\bf u} \mathrel{\tau} {\bf v}$ whenever $\mathbb V$ satisfies ${\bf u} \approx {\bf v}$. 
Notice that if $W \subseteq \mathfrak A^{\ast}$ forms a single $\tau$-class, then $W$ is stable with respect to $\mathbb V$ if and only if every word in ${\bf u} \in W$ is a $\tau$-term for $\mathbb V$.
If $\tau$ is an equivalence relation on $\mathfrak A^{\ast}$ and $\mathbf w\in\mathfrak A^{\ast}$, then $[\mathbf w]_\tau$ denotes the $\tau$-class of $\mathbf w$.

We use regular expressions to describe sets of words, in particular, congruence classes. 
For brevity, we put $a^+=\{a\}^+$ and $a^\ast=\{a\}^\ast$ for any $a\in\mathfrak A$.
For example, $[atb^2a ]_\beta = a^+tbb^+a \{a,b\}^{\ast} \cup  a^+tb^+a^+b \{a,b\}^{\ast}$, where $\beta$ be the fully invariant congruence of $\mathbb E^1$.

\begin{fact}[\mdseries{\cite[Corollary~3.6,~ Theorem~4.1(iv)]{Sapir-21}; \cite[Example~4.2]{Sapir-23}}]
\label{F: E2}
Given a monoid variety $\mathbb V$ we have:
\begin{itemize}
\item[\textup{(i)}] $\mathbb V$ contains $\mathbb A_0^1$ if and only if  the set $aa^+bb^+$ is stable with respect to $\mathbb V$;
\item[\textup{(ii)}] $\mathbb V$ contains $\mathbb E^1 \{\sigma_2\}$ if and only if  $\beta$-class $[at b^2 a]_{\beta}$ is stable with respect to $\mathbb V$.\qed
\end{itemize}
\end{fact}

The next statement is a special case of Lemma~2.4 in \cite{Sapir-23}.

\begin{fact}      
\label{F: abt}
Suppose that $[atb^2a ]_\beta$ is  stable with respect to a monoid variety $\mathbb V$.  Then  ${\bf u} \in \{a,b, t\}^{\ast}$  is $\beta$-term for $\mathbb V$ in each of the following cases:
\begin{itemize}
\item[\textup{(i)}]$\con ({\bf u}) \subset \{a,b, t\}$;
\item[\textup{(ii)}] ${\bf u} = {\bf a}t{\bf b}$  such that  ${\bf a}, {\bf b}  \in \{a,b \}^{\ast}$
and either $\bf a$ or $\bf b$ contains no occurrences of  $a$.\qed
\end{itemize}
\end{fact}

\section{Varieties of idempotent monoids}
\label{sec: idempotent}

A complete description of the subvariety lattice of the variety of all idempotent monoids $\mathbb B^1$ was given by Wismath~\cite{Wismath-86}; see also~\cite[Section 5.5]{Almeida-94}.
To describe this lattice, for $n \ge 2$, define
\begin{align*}
\mathbb R_n^1 = \var\{x\approx x^2,\,\mathbf R_n\approx\mathbf S_n\},\\
\mathbb L_n^1 = \var\{x\approx x^2,\,\overline{\mathbf R_n}\approx\overline{\mathbf S_n}\},
\end{align*}
where
\begin{align*}
\mathbf R_n= &
\begin{cases}
x_2x_1&\text{ for }n=2,\\
x_1x_2x_3&\text{ for }n=3,\\
\mathbf R_{n-1}x_n&\text{ for even }n\ge 4,\\
x_n\mathbf R_{n-1}&\text{ for odd }n\ge 5
\end{cases}
\\
\text{and}\enskip\mathbf S_n= &
\begin{cases}
x_1x_2x_1&\text{ for }n=2,\\
x_1x_2x_3x_1x_3x_2x_3&\text{ for }n=3,\\
\mathbf S_{n-1}x_n\mathbf R_n&\text{ for even }n\ge 4,\\
\mathbf R_nx_n\mathbf S_{n-1}&\text{ for odd }n\ge 5,
\end{cases}
\end{align*}
while $\overline{\mathbf R_n}$ and $\overline{\mathbf S_n}$ are the mirror images of $\mathbf R_n$ and $\mathbf S_n$, respectively.
Let $\mathbb T$ denote the variety of trivial monoids.

\begin{prop}[\mdseries{\cite[Proposition~4.7]{Wismath-86}}]
\label{P: struct of L(B)}
The subvariety lattice of  $\mathbb B^1$ is given in Fig.~\ref{L(B)}.\qed
\end{prop}

\begin{figure}[htb]
\unitlength=1mm
\linethickness{0.4pt}
\begin{center}
\begin{picture}(60,94)
\put(10,23){\circle*{1.33}}
\put(10,43){\circle*{1.33}}
\put(10,63){\circle*{1.33}}
\put(20,3){\circle*{1.33}}
\put(20,13){\circle*{1.33}}
\put(20,33){\circle*{1.33}}
\put(20,53){\circle*{1.33}}
\put(20,73){\circle*{1.33}}
\put(20,89){\circle*{1.33}}
\put(30,23){\circle*{1.33}}
\put(30,43){\circle*{1.33}}
\put(30,63){\circle*{1.33}}
\gasset{AHnb=0,linewidth=0.4}
\drawline(20,3)(20,13)(10,23)(30,43)(10,63)(22,75)
\drawline(20,13)(30,23)(10,43)(30,63)(18,75)
\put(20,84){\makebox(0,0)[cc]{$\vdots$}}
\put(31,43){\makebox(0,0)[lc]{$\mathbb R_3^1$}}
\put(9,43){\makebox(0,0)[rc]{$\mathbb L_3^1$}}
\put(9,63){\makebox(0,0)[rc]{$\mathbb R_4^1$}}
\put(31,63){\makebox(0,0)[lc]{$\mathbb L_4^1$}}
\put(20,92){\makebox(0,0)[cc]{$\mathbb B^{1}$}}
\put(9,23){\makebox(0,0)[rc]{$\mathbb R_2^1$}}
\put(31,23){\makebox(0,0)[lc]{$\mathbb L_2^1$}}
\put(21,12){\makebox(0,0)[lc]{$\mathbb M(\{1\})=\var\{x\approx x^2,  xy\approx yx\}$}}
\put(22,33){\makebox(0,0)[lc]{$\mathbb L_2^1\vee \mathbb R_2^1=\var\{x\approx x^2,  \overline{{\bf R}_2} t {\bf R}_2 \approx {\bf S}_2 t \overline{{\bf S}_2}\}$}}
\put(22,53){\makebox(0,0)[lc]{$\mathbb L_3^1\vee \mathbb R_3^1=\var\{x\approx x^2,  {\bf R}_3 t \overline{{\bf R}_3} \approx {\bf S}_3 t \overline{{\bf S}_3} \}$}}
\put(22,73){\makebox(0,0)[lc]{$\mathbb L_4^1\vee \mathbb R_4^1=\var\{x\approx x^2,  \overline{{\bf R}_4} t {\bf R}_4 \approx {\bf S}_4 t \overline{{\bf S}_4}\}$}}
\put(20,0){\makebox(0,0)[cc]{$\mathbb T$}}
\end{picture}
\end{center}
\caption{The subvariety lattice of $\mathbb B^1$}
\label{L(B)}
\end{figure}
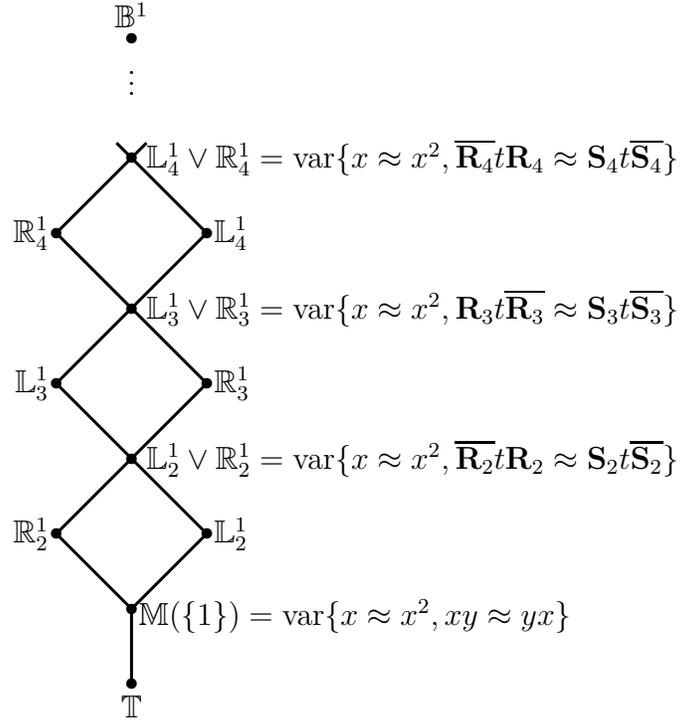

Proposition~\ref{P: struct of L(B)} gives us a freedom in choosing the defining identities for varieties of idempotent monoids as follows.

\begin{cor} 
\label{C: struct of L(B)}
Let $k \ge 2$, ${\bf u} \approx {\bf v}$ be any identity which holds in $\mathbb R_k^1$ and fails in $\mathbb L_k^1$, and ${\bf w} \approx {\bf p}$ any identity which holds in  $\mathbb L^1_k \vee \mathbb R^1_k$ but  fails in $\mathbb L^1_{k+1}$ and in  $\mathbb R^1_{k+1}$. Then
\[
\begin{aligned}
&\mathbb R^1_k = \var  \{x \approx x^2,  {\bf u} \approx {\bf v}\},\\
&\mathbb R^1_k \vee \mathbb L^1_k   = \var  \{x \approx x^2,  {\bf w} \approx {\bf p}\} =  \mathbb R^1_{k+1} \wedge \mathbb L^1_{k+1}. \quad\quad\quad\quad\quad \quad\quad\quad\quad\quad\quad\quad\quad\qed
\end{aligned}
\] 
\end{cor}

We use $\Id(\mathbb V)$ to denote the set of all identities of the variety $\mathbb V$.
Let $\ell({\bf u})$ denote the maximal prefix of $\bf u$ which contains all letters in $\bf u$ but one.
Dually,  $r({\bf u})$ denotes the maximal suffix of $\bf u$ which contains all letters in $\bf u$ but one.
The following result on the variety of all idempotent semigroups $\mathbb B$ is useful here because $\Id(\mathbb B)=\Id(\mathbb B^1)$.

\begin{fact}[\mdseries{\cite{Green-Rees-52}}]
\label{F: x=xx}
An identity ${\bf u} \approx {\bf v}$ holds in $\mathbb B^{1}$ if and only if  $\con({\bf u}) = \con ({\bf v})$ and the identities  $\ell({\bf u}) \approx \ell({\bf v})$ and $r({\bf u}) \approx r({\bf v})$ hold in $\mathbb B^{1}$.\qed
\end{fact}

The next well-known claim readily follows from Fact~\ref{F: x=xx} or from the classical result \cite[Exercise 4.2.1]{Clifford-Preston-61} that an idempotent semigroup is a semilattice of rectangular bands.

\begin{cor}
\label{C: uwu} 
If $\con({\bf w}) \subseteq \con({\bf u})=\con({\bf v})$, then $\mathbb B^{1}$ satisfies ${\bf u}{\bf v} \approx {\bf u w v}$.\qed
\end{cor}

\begin{fact}[\mdseries{\cite[Lemma~3.20]{Fennemore-69}}]
\label{F: R31}
An identity ${\bf u} \approx {\bf v}$ holds in $\mathbb R_3^1$ if and only if ${\bf u} \approx {\bf v}$ holds in $\mathbb L^1_2$ and $r({\bf u}) \approx r({\bf v})$ holds in $\mathbb R^1_3$.\qed
\end{fact}

\begin{cor} 
\label{C: adjacent-L}
For any $\mathbf w\in\mathfrak A^\ast$, the variety $\mathbb R_3^1$ satisfies the identity $\mathbf w\approx\ini(\mathbf w)\mathbf w$.
\end{cor}

\begin{proof} 
Clearly, $\ini(\mathbf w)= \ini(\ini(\mathbf w)\mathbf w)$ and $r(\mathbf w)= r(\ini(\mathbf w)\mathbf w)$.
Now Facts~\ref{F: L2}(i) and~\ref{F: R31} apply.
\end{proof}

We use $_{i{\bf u}}x$ to refer to the $i$th from the left occurrence of $x$ in a word ${\bf u}$.
We use $_{\ell {\bf u}}x$ to refer to the last occurrence of $x$ in ${\bf u}$.  
If $x$ is simple in $\bf u$ then we use $_{{\bf u}}x$ to denote the only occurrence of $x$ in $\bf u$.
If  the $i$th occurrence of $x$ precedes  the $j$th occurrence of $y$ in a word $\bf u$, we write $({_{i{\bf u}}x}) <_{\bf u} ({_{j{\bf u}}y})$.

If $\{z, y\} \subseteq \con({\bf u})$ and $({_{\ell{\bf u}}z}) <_{\bf u} ({_{\ell{\bf u}}y})$, we use ${_{(1z){\bf u}}y}$  to denote the first occurrence of $y$ in $\mathbf u$ after  ${_{\ell{\bf u}}z}$.
We say that a triple of pairwise distinct letters $\{x,y,z \} \subseteq \con({\bf u})$ is {\em $\mathbb R_3^1$-stable} in  ${\bf u} \approx {\bf v}$ if  
\[
(_{\ell{\bf u}}z) <_{\bf u} ({ _{(1z){\bf u}}y}) <_{\bf u}  ({ _{(1z){\bf u}}x}) \Leftrightarrow 
 (_{\ell{\bf v}}z) <_{\bf v} ({ _{(1z){\bf v}}y}) <_{\bf v}  ({ _{(1z){\bf v}}x}). 
 \]
Otherwise, we say that $\{x, y, z\}$ is {\em $\mathbb R_3^1$-unstable} in ${\bf u} \approx {\bf v}$.

\begin{fact} 
\label{F: R3}
An identity ${\bf u} \approx {\bf v}$ holds in $\mathbb R_3^1$ if and only if ${\bf u} \approx {\bf v}$ holds in $\mathbb L_2^1 \vee \mathbb R_2^1$ and each triple of pairwise distinct letters  $\{x, y, z\} \subseteq \con({\bf u})$ is $\mathbb R_3^1$-stable in ${\bf u} \approx {\bf v}$.
\end{fact}

\begin{proof} 
Suppose that  ${\bf u} \approx {\bf v}$ holds in $\mathbb R_3^1$.
According to Proposition~\ref{P: struct of L(B)}, ${\bf u} \approx {\bf v}$ is satisfied by $\mathbb L_2^1 \vee \mathbb R_2^1$.
Take some triple of pairwise distinct letters  $\{x, y, z\} \subseteq \con({\bf u})$.
Without loss of generality we may assume that $(_{\ell{\bf u}}z) <_{\bf u} ({ _{(1z){\bf u}}y}) <_{\bf u}  ({ _{(1z){\bf u}}x})$.
Since ${\bf u} \approx {\bf v}$ holds in $\mathbb R_2^1$, Fact~\ref{F: L2}(ii) implies that $(_{\ell{\bf v}}z) <_{\bf u} ({ _{\ell{\bf v}}y})$ and $(_{\ell{\bf v}}z) <_{\bf u}  ({ _{\ell{\bf v}}x})$. 
If $(_{\ell{\bf v}}z) <_{\bf v} ({ _{(1z){\bf v}}x}) <_{\bf v}  ({ _{(1z){\bf v}}y})$, then $r({\bf u} (x, y, z))$ begins with $y$, but $r({\bf v} (x, y, z))$ begins with $x$. 
Hence ${\bf u} (x, y, z) \approx {\bf v}(x, y, z)$ fails in $\mathbb R_3^1$ by Facts~\ref{F: L2}(i) and~\ref{F: R31}.
To avoid a contradiction, the triple $\{x, y, z\} \subseteq \con({\bf u})$ must be $\mathbb R_3^1$-stable
in ${\bf u} \approx {\bf v}$.

Conversely, suppose that ${\bf u} \approx {\bf v}$ holds in $\mathbb L_2^1 \vee \mathbb R_2^1$ and
each triple of pairwise distinct letters  $\{x, y, z\} \subseteq \con({\bf u})$ is $\mathbb R_3^1$-stable.
In view of Fact~\ref{F: L2}(ii), $\fin ({\bf u})=\fin ({\bf v}) =  x_nx_{n-1} \cdots x_1$ for some $n \ge 1$. 
We use induction on $n$.
If $n < 3$ then ${\bf u} \approx {\bf v}$ holds in $\mathbb R_3^1$ by Fact~\ref{F: R31}. 
Suppose that $n \ge 3$. 
Clearly, each triple of pairwise distinct letters  $\{x, y, z\} \subseteq \con(r({\bf u}))$ is $\mathbb R_3^1$-stable in $r({\bf u}) \approx r({\bf v})$.
Since $\fin(r({\bf u}))=\fin(r({\bf v}))=x_{n-1} x_{n-2}\cdots x_1$, Fact~\ref{F: L2}(ii) implies that $r({\bf u}) \approx r({\bf v})$ holds in $\mathbb R_2^1$.
Further, for each distinct $y,z\in\con(r({\bf u}))$, the triple $\{x_n, y, z\}$ is $\mathbb R_3^1$-stable in ${\bf u} \approx {\bf v}$.
Hence $\ini(r({\bf u}) )=\ini(r({\bf v}))$.
Then $r({\bf u}) \approx r({\bf v})$ holds in $\mathbb L_2^1$ by Fact~\ref{F: L2}(i).
By the induction assumption, $r({\bf u}) \approx r({\bf v})$ holds in $\mathbb R_3^1$.
Now Fact~\ref{F: R31} applies, yielding that ${\bf u} \approx {\bf v}$ is satisfied by $\mathbb R_3^1$.
\end{proof}

\begin{ex} 
\label{E: struct of L(B)}
\quad
\begin{itemize}
\item[\textup{(i)}] $\mathbb R_2^1 \vee \mathbb L_2^1=\var\{x \approx x^2, \, xszx \approx xsxzx\} = \mathbb R_3^1 \wedge \mathbb L_3^1$;
\item[\textup{(ii)}] $\mathbb R_3^1=\var\{x \approx x^2, \, xyt xy \approx xyx t xy\}$.
\end{itemize}
\end{ex}

\begin{proof} 
(i) The identity  $xszx \approx xsxzx$ holds in $\mathbb R_2^1\vee\mathbb L_2^1$ by Fact~\ref{F: L2}.
On the other hand, it fails in $\mathbb R_3^1$ and in $\mathbb L_3^1$ by Fact~\ref{F: R3} and its dual.
Therefore, $\mathbb R_2^1 \vee \mathbb L_2^1=\var\{x \approx x^2, \, xszx \approx xsxzx\} = \mathbb R_3^1 \wedge \mathbb L_3^1$ by
Corollary~\ref{C: struct of L(B)}.

\smallskip

(ii) The identity  $xyt xy \approx xyx t xy$ holds in $\mathbb R_3^1$ but fails in $\mathbb L_3^1$ by Fact~\ref{F: R3} and its dual.
Therefore, $\mathbb R_3^1=\var\{x \approx x^2, \, xyt xy \approx xyx t xy\}$ by Corollary~\ref{C: struct of L(B)}.
\end{proof}

\section{The lattice of subvarieties of $\mathbb Q^1 \vee \mathbb B^1$} 
\label{sec: QB}

\begin{lemma} 
\label{L: xyxy}
\quad
\begin{itemize}
\item[\textup{(i)}] $\mathbb Q^1 \vee \mathbb B^1 = \var \{xtx \approx xtx^2 \approx x^2tx, (xy)^2 \approx x^2y^2\}$;
\item[\textup{(ii)}] An identity ${\bf u} \approx {\bf v}$  holds in   $\mathbb Q^1 \vee \mathbb B^1$   if and only if ${\bf u} \approx {\bf v}$ holds in $\mathbb Q^1$ and the corresponding blocks of $\bf u$ and $\bf v$ are equivalent modulo $x \approx x^2$.
\end{itemize}
\end{lemma}

\begin{proof} 
(i)  Evidently, $x\approx x^2$ implies $xtx \approx xtx^2\approx x^2tx$ and $(xy)^2 \approx x^2y^2$.
It follows from Fact~\ref{F: Q} that $\mathbb Q^1$ satisfies $xtx \approx xtx^2\approx x^2tx$ and $(xy)^2 \approx x^2y^2$ as well.
Therefore, $\mathbb Q^1 \vee \mathbb B^1 \subseteq \var \{xtx \approx xtx^2\approx x^2tx, (xy)^2 \approx x^2y^2\}$. 

Conversely, if ${\bf u} \approx {\bf v}$ is an identity of $\mathbb Q^1$, then,
by Fact~\ref{F: Q}, ${\bf u}= \mathbf a_0 \prod_{i=1}^m (t_i\mathbf a_i)$ and  ${\bf v} = \mathbf b_0 \prod_{i=1}^m (t_i\mathbf b_i)$ for some $m \ge 0$, $ \simp({\bf u})= \simp({\bf v})=\{t_1,t_2,\dots, t_m\}$, and $\con({\bf a}_i) = \con({\bf b}_i)$ for each $0 \le i \le m$. 
If, in addition, ${\bf u} \approx {\bf v}$ holds in $\mathbb B^1$, then according to Fact~\ref{F: x=xx}, the corresponding blocks $\mathbf a_i$ and $\mathbf b_i$ of $\bf u$ and $\bf v$, respectively, must be equivalent modulo $x \approx x^2$.
Then  ${\bf u} \approx {\bf v}$ can be derived from $\{xtx \approx xtx^2\approx x^2tx, (xy)^2 \approx x^2y^2\}$.
Therefore, $\mathbb Q^1 \vee \mathbb B^1 = \var \{xtx \approx xtx^2\approx x^2tx, (xy)^2 \approx x^2y^2\}$. 

Part (ii) readily follows from Part (i) and its proof. 
\end{proof} 

\begin{obs} 
\label{O: B} 
Let $\mathbb V$ be a subvariety of $\mathbb Q^1 \vee \mathbb B^{1}$ and ${\bf u} \approx {\bf v}$ an identity of $\mathbb V \wedge \mathbb B^{1}$ such that $\con({\bf u}) = \mul ({\bf u})$. Then ${\bf u} \approx {\bf v}$ holds on $\mathbb V$.
\end{obs}

\begin{proof}  
Take some monoid $M \in \mathbb V$.
According to Lemma~\ref{L: xyxy},  $\{xtx \approx xtx^2\approx x^2tx, (xy)^2 \approx x^2y^2\}$ holds in $M$.
Since $M$ satisfies $x^2y^2 \approx (xy)^2$, the set of all idempotents $E(M)$ of $M$ forms a submonoid of $M$. 
Further, since $M$ satisfies $x^2 \approx x^3$, the square of every element in $M$ is an idempotent. 
Finally, since $E(M)$ satisfies ${\bf u} \approx {\bf v}$, the monoid $M$ satisfies the identity obtained from ${\bf u} \approx {\bf v}$ by replacing every letter $x \in \con({\bf u})$ by $x^2$.
Since $\con({\bf u}) = \mul ({\bf u})$, the resulting identity is equivalent to ${\bf u} \approx {\bf v}$ modulo $xtx \approx xtx^2 \approx x^2tx$.
\end{proof}

Let $k \ge 2$ and ${\bf u} \approx {\bf v}$ be any identity which holds in $\mathbb R_k^1$ and fails in $\mathbb L_k^1$ with $\con({\bf u}) = \mul({\bf u})$.
Denote
\[
\mathbb E^1_k = \var  \{xtx \approx xtx^2 \approx x^2tx, x^2y^2 \approx (xy)^2,  {\bf u} \approx {\bf v}\}.\]
For example,
\begin{align*}
&\overline{\mathbb E^1} = \mathbb E^1_2 =  \var \{xtx \approx xtx^2 \approx x^2tx,\, x^2y^2 \approx (xy)^2,  y^2 x^2 \approx xy^2x  \},\\
&\mathbb E^1 = \overline{\mathbb E^1_2} =  \var \{xtx \approx xtx^2 \approx x^2tx,\, x^2y^2 \approx (xy)^2,  x^2y^2 \approx xy^2x  \}.
\end{align*}
See also Example~\ref{E: EE}(ii) below.

The following statement implies that for each $k \ge 2$ the monoid variety $\mathbb E^1_k$ is well-defined.

\begin{fact}  
\label{F: Ek} 
For each $k \ge 2$ the variety  $\mathbb E^1_k$ is the largest subvariety of  $\mathbb Q^1 \vee \mathbb B^{1}$ which does not contain $\mathbb L_k^1$.
\end{fact}

\begin{proof} 
Let $\mathbb V$ be a subvariety of $\mathbb Q^1 \vee \mathbb B^{1}$ which does not contain $\mathbb L_k^1$. 
Then in view of Proposition~\ref{P: struct of L(B)}, every idempotent subvariety of $\mathbb V$ is contained in $\mathbb R_k^1 = \var\{x \approx x^2, \sigma\}$, where the identity $\sigma$ holds in $\mathbb R_k^1$ and fails in $\mathbb L_k^1$.
Then Observation~\ref{O: B} implies that $\mathbb V$ satisfies every identity ${\bf u} \approx {\bf v}$  with $\con({\bf u}) = \mul({\bf u})$ that holds in $\mathbb R_k^1$ and fails in $\mathbb L_k^1$.
Therefore,  $\mathbb V$ is a subvariety of  $\mathbb E^1_k$.
\end{proof}

Given $k \ge 2$ an identity ${\bf u} \approx {\bf v}$ is  called \textit{block-$\mathbb R_k^1$-balanced}  if
$Q^1$ satisfies ${\bf u} \approx {\bf v}$ and the corresponding blocks in  $\bf u$ and $\bf v$ form
an identity which holds in $\mathbb R_k^1$. We say that a property of identities~(P) is \textit{derivation-stable} if an identity ${\bf u} \approx {\bf v}$ satisfies property~(P) whenever ${\bf u} \approx {\bf v}$ follows from $\Sigma$ such that every identity in $\Sigma$ satisfies property~(P). 
It is easy to see that the property of being a  block-$\mathbb R_k^1$-balanced identity is derivation-stable. The next observation generalizes Fact~\ref{F: E} and Lemma~\ref{L: xyxy}(ii).

\begin{obs} 
\label{O: rb}  
Given $k \ge 2$ an identity holds in $\mathbb E^1_k$ if and only if it is block-$\mathbb R_k^1$-balanced.
\end{obs} 

\begin{proof}  
Fix some identity ${\bf u} \approx {\bf v}$ which holds in $\mathbb R_k^1$ and fails in $\mathbb L_k^1$ with $\con({\bf u}) = \mul({\bf u})$.
Since every identity in $\{xtx \approx xtx^2 \approx x^2tx, x^2y^2 \approx (xy)^2,  {\bf u} \approx {\bf v}\}$ is  block-$\mathbb R_k^1$-balanced and this property of identities is derivation-stable, every identity of $\mathbb E^1_k$ is block-$\mathbb R_k^1$-balanced.

Conversely, in view of Proposition~\ref{P: struct of L(B)}, if  ${\bf U} \approx {\bf V}$ is block-$\mathbb R_k^1$-balanced, then we can derive ${\bf U} \approx {\bf V}$ from  $\{xtx \approx xtx^2 \approx x^2tx, x^2y^2 \approx (xy)^2,  {\bf u} \approx {\bf v}\}$ by applying ${\bf u} \approx {\bf v}$ within the blocks.
\end{proof}

The next proposition shows that the varieties in $\{\mathbb T,\mathbb M(\{1\}), \mathbb E^1_k, \overline{\mathbb E_k^1} \mid k \ge 2\}$ generate a lattice isomorphic to the lattice of  varieties of idempotent monoids. 
Compare it with Corollary~\ref{C: struct of L(B)}.

\begin{prop} 
\label{P: EE}
For each $k \ge 2$ we have
\[
\mathbb E^1_k \vee \overline{\mathbb E^1_k}   = \var  \{xtx \approx xtx^2 \approx x^2tx, x^2y^2 \approx (xy)^2,  {\bf u} \approx {\bf v}\} =  \mathbb E^1_{k+1} \wedge \overline{\mathbb E^1_{k+1}},
\] 
where  ${\bf u} \approx {\bf v}$ holds in  $\mathbb L^1_k \vee \mathbb R^1_k$, fails in $\mathbb L^1_{k+1}$ and in  $\mathbb R^1_{k+1}$, and $\con({\bf u}) = \mul({\bf u})$.
\end{prop}

\begin{proof} 
Observation~\ref{O: rb} and its dual imply that
\[
\mathbb E^1_k \vee \overline{\mathbb E^1_k}  \subseteq \var  \{xtx \approx xtx^2 \approx x^2tx, x^2y^2 \approx (xy)^2,  {\bf u} \approx {\bf v}\}.
\]      
Conversely, let ${\bf U} \approx {\bf V}$ be an identity of  $\mathbb E^1_k \vee \overline{\mathbb E^1_k}$. Then it holds in $\mathbb Q^1$ and its corresponding blocks form the identities which hold both in $\mathbb R_k^1$ and in $\mathbb L_k^1$ by Observation~\ref{O: rb} and its dual.
So, we can derive ${\bf U} \approx {\bf V}$ from  $\{xtx \approx xtx^2 \approx x^2tx, x^2y^2 \approx (xy)^2,  {\bf u} \approx {\bf v}\}$ by applying ${\bf u} \approx {\bf v}$ within the blocks. 
Therefore,
\[
\mathbb E^1_k \vee \overline{\mathbb E^1_k}  = \var  \{xtx \approx xtx^2 \approx x^2tx, x^2y^2 \approx (xy)^2,  {\bf u} \approx {\bf v}\}.
\]      
Since
\[
\mathbb E^1_{k+1} \wedge \overline{\mathbb E^1_{k+1}} = \var  \{xtx \approx xtx^2 \approx x^2tx, x^2y^2 \approx (xy)^2,  {\bf u}_1 \approx {\bf v}_1, {\bf u}_2 \approx {\bf v}_2 \},
\]
where ${\bf u}_1 \approx {\bf v}_1$  holds in $\mathbb R_{k+1}^1$ and fails in $\mathbb L_{k+1}^1$ with $\con({\bf u}_1) = \mul({\bf u}_1)$, ${\bf u}_2 \approx {\bf v}_2$  holds in $\mathbb L_{k+1}^1$ and fails in $\mathbb R_{k+1}^1$ with $\con({\bf u}_2) = \mul({\bf u}_2)$, Corollary~\ref{C: struct of L(B)} implies that the sets of identities $\{{\bf u}_1 \approx {\bf v}_1, {\bf u}_2 \approx {\bf v}_2 \}$ and $\{{\bf u} \approx {\bf v}\}$ are equivalent modulo $x \approx x^2$. 
Therefore,
\[
\mathbb E^1_{k+1} \wedge \overline{\mathbb E^1_{k+1}} = \var  \{xtx \approx xtx^2 \approx x^2tx, x^2y^2 \approx (xy)^2,  {\bf u} \approx {\bf v} \},
\]
as required.
\end{proof}


\begin{cor} 
\label{C: EEk}
For each $k \ge 2$, $\mathbb E^1_k \vee \overline{\mathbb E^1_k}$ is the largest subvariety of  $\mathbb Q^1 \vee \mathbb B^{1}$ which contains neither $\mathbb L_{k+1}^1$ nor $\mathbb R_{k+1}^1$.
\end{cor}

\begin{proof} 
Recall that,  by Lemma~\ref{L: xyxy}(i), $\mathbb Q^1 \vee \mathbb B^{1}= \var  \{xtx \approx xtx^2 \approx x^2tx, x^2y^2 \approx (xy)^2 \}$.
Let $\mathbb V$ be a subvariety of $\mathbb Q^1 \vee \mathbb B^{1}$ which contains neither $\mathbb L_{k+1}^1$ nor $\mathbb R_{k+1}^1$. 
Then $\mathbb V$ is contained in $\mathbb E^1_{k+1} \wedge \overline{\mathbb E^1_{k+1}}$ by Fact~\ref{F: Ek} and its dual.
Since $\mathbb E^1_k \vee \overline{\mathbb E^1_k}   =  \mathbb E^1_{k+1} \wedge \overline{\mathbb E^1_{k+1}}$ by Proposition~\ref{P: EE}, 
$\mathbb V$ is contained in $\mathbb E^1_k \vee \overline{\mathbb E^1_k}$.
\end{proof}

Compare the following with Example~\ref{E: struct of L(B)}.

\begin{ex} 
\label{E: EE}
\quad
\begin{itemize}
\item[\textup{(i)}] $\mathbb E^1 \vee \overline{\mathbb E^1}=\var\{xtx \approx xtx^2 \approx x^2tx,\,x^2y^2 \approx (xy)^2,\, xs^2z^2x \approx xs^2xz^2x\}= \mathbb E_3^1 \wedge \overline{\mathbb E_3^1}$.
\item[\textup{(ii)}] $\mathbb E^1_3 =  \var \{xtx \approx xtx^2 \approx x^2tx,\, x^2y^2 \approx (xy)^2,  xyt^2xy \approx xyx t^2 xy\}$.
\end{itemize}
\end{ex}

\begin{proof} 
(i) The identity  $xs^2z^2x \approx xs^2xz^2x$ holds in $\mathbb R_2^1\vee\mathbb L_2^1$ by Fact~\ref{F: L2}.
On the other hand, it fails in $\mathbb R_3^1$ and in $\mathbb L_3^1$ by Fact~\ref{F: R3} and its dual.
Therefore, $\mathbb E^1 \vee \overline{\mathbb E^1}=\var\{xtx \approx xtx^2 \approx x^2tx,\,x^2y^2 \approx (xy)^2,\, xs^2z^2x \approx xs^2xz^2x\}= \mathbb E_3^1 \wedge \overline{\mathbb E_3^1}$ by  Proposition~\ref{P: EE}.

\smallskip

(ii) The identity  $xyt^2xy \approx xyx t^2 xy$ holds in $\mathbb R_3^1$ by Facts~\ref{F: L2} and~\ref{F: R3}.
On the other hand, it fails in $\mathbb L_3^1$ by the dual of Fact~\ref{F: R3}.
Therefore, $\mathbb E^1_3 =  \var \{xtx \approx xtx^2 \approx x^2tx,\, x^2y^2 \approx (xy)^2,  xyt^2xy \approx xyx t^2 xy\}$ by the very definition.
\end{proof}

\begin{obs} 
\label{O: same}
The following pairs of sets of identities are equivalent within the variety $\mathbb Q^1 \vee \mathbb B^{1}$:
\begin{itemize}
\item[\textup{(i)}] $\{ xy t yx \approx xy t xyx\}$ and $\{ xy t xy \approx xy t yxy\}$;
\item[\textup{(ii)}] $\{ xy t yx \approx xyx t xyx\}$ and $\{ xy t yx \approx xy t xyx,    xy t yx \approx xyx t yx \}$;
\item[\textup{(iii)}] $\{x^2y^2 \cdot {\bf a}_1 x {\bf a}_2 y \cdots {\bf a}_n x \approx xy^2x \cdot {\bf a}_1 x {\bf a}_2 y \cdots {\bf a}_n x\}$ and $\{ x^2y^2 \cdot {\bf a}_1 y {\bf a}_2 x \cdots {\bf a}_n y \approx xy^2x \cdot {\bf a}_1 y {\bf a}_2 x \cdots {\bf a}_n y\}$, where $n \ge 0$ and $\{x,y\} \cap \con ({\bf a}_1 \dots {\bf a}_n) =\emptyset$.
\item[\textup{(iv)}] $\{ xy t^2 y s x \approx xyx t^2 y s x\}$ and $\{ xy t^2 x s y \approx xyx t^2 x s y\}$;
\item[\textup{(v)}] $\{ x^2y^2 t y  \approx xy^2x t y\}$ and $\{ x^2y^2 t x \approx xy^2x t x\}$.
\end{itemize}
\end{obs}

\begin{proof} 
(i) If we replace $t$ by $xt$ and multiply  $xy t yx \approx xy t xyx$ by $y$ on the left, then we obtain an identity which is equivalent within $\mathbb Q^1 \vee \mathbb B^{1}$ to $yx tyx \approx yxt xyx$.
Swapping $x$ and $y$ we obtain $xy t xy \approx xy t yxy$.

Conversely, if we replace $t$ by $xt$ and multiply  $xy t xy \approx xy t yxy$ by $y$ on the left, then we obtain an identity which is equivalent within $\mathbb Q^1 \vee \mathbb B^{1}$ to $yx t xy \approx yxt yxy$.
Swapping $x$ and $y$ we obtain $xy t yx \approx xy t xyx$.

\smallskip

(ii) First, $xy t yx \approx xyx t xyx$ implies $xy t yx \approx xy t xyx$ within $\mathbb Q^1 \vee \mathbb B^{1}$ because of the following deduction:
\[
xy t yx \approx (xy) xy t (yx) \stackrel{\mathbb Q^1 \vee \mathbb B^{1}}\approx (xyx) xy t (xyx) \approx xy t xyx.
\] 
Similarly, $xy t yx \approx xyx t xyx$ implies $xy t yx \approx xyx t yx$ within $\mathbb Q^1 \vee \mathbb B^{1}$. 

Reversely, $\{ xy t yx \approx xy t xyx, xy t yx \approx xyx t yx \}$ implies $xy t yx \approx xyx t xyx$ within $\mathbb Q^1 \vee \mathbb B^{1}$ as follows:
\[ 
xy t yx \approx   xy t xyx \approx xy (xy) t x (yx) \approx     xy (xyx) t x (yx)  \approx   xyx t xyx,
\]
and we are done.

\smallskip

(iii) Multiplying both sides of  $x^2y^2 \cdot {\bf a}_1 x {\bf a}_2 y \cdots {\bf a}_n x \approx xy^2x \cdot {\bf a}_1 x {\bf a}_2 y \cdots {\bf a}_n x$ by $y$ on the left and using $\{xtx \approx xtx^2 \approx x^2tx, x^2y^2 \approx (xy)^2 \}$ gives an identity which is equivalent to $x^2y^2 \cdot {\bf a}_1 y {\bf a}_2 x \cdots {\bf a}_n y \approx xy^2x \cdot {\bf a}_1 y {\bf a}_2 x \cdots {\bf a}_n y$ modulo swapping $x$ and $y$.

\smallskip

Parts (iv) and (v) readily follow from Part (iii).
\end{proof}

\section{The variety $\mathbb E^1 \vee \overline{\mathbb E^1}$ is HFB} 
\label{sec: EEhfb}

\begin{lemma} 
\label{L: fb} 
Let $\mathbb V$, $\mathbb S$ and $\mathbb P$ be three varieties such that $\mathbb P \subseteq \mathbb V$ and $\mathbb P \subseteq \mathbb S$.
Let $\Sigma$ be a set of identities such that $\mathbb V$ satisfies $\Sigma$.

Let $\Dist(\mathbb P \rightarrow  \mathbb S)$ be a function which associates with each identity ${\bf u} \approx {\bf v}$ of $\mathbb P$ a set $\Dist(\mathbb P \rightarrow \mathbb S)({\bf u} \approx {\bf v})$ so that the set $\Dist(\mathbb P \rightarrow \mathbb S)({\bf u} \approx {\bf v})$ is empty if and only if ${\bf u} \approx {\bf v}$ holds on $\mathbb S$.

Suppose that for every identity ${\bf u} \approx {\bf v}$ of $\mathbb V$ which fails on $\mathbb S$, one can find a word ${\bf u}_1$ such that $\Sigma$ implies ${\bf u} \approx {\bf u}_1$ and $|\Dist(\mathbb P \rightarrow  \mathbb S)({\bf u}_1 \approx {\bf v})| < |\Dist(\mathbb P \rightarrow \mathbb S)({\bf u} \approx {\bf v})|$.

Then every identity of $\mathbb V$ can be derived from $\Sigma$ and an identity of $\mathbb V \vee \mathbb S$.
\end{lemma}

\begin{proof} 
Let $\Dist(\mathbb V \rightarrow  \mathbb S)$ be the restriction of $\Dist(\mathbb P \rightarrow  \mathbb S)$ to the identities of $\mathbb V$.
Then by Lemma~3.1 in \cite{Sapir-15} and its proof, for every identity ${\bf u} \approx {\bf v}$ of $\mathbb V$ there is a derivation 
\[
{\bf u}={\bf u}_0 \approx {\bf u}_1 \approx {\bf u}_2 \approx \dots \approx {\bf u}_k \approx {\bf v},
\]
such that for each $i \ge 0$, the identity  ${\bf u}_i \approx {\bf u}_{i+1}$ follows from $\Sigma$ and
${\bf u}_k \approx {\bf v}$ holds on $\mathbb S$. 
Since $\mathbb V$ satisfies $\Sigma$, the identity ${\bf u}_k \approx {\bf v}$ holds on $\mathbb V \vee \mathbb S$. 
\end{proof}

Let ${\bf u} \approx {\bf v}$ be an identity of $\mathbb Q^1$. 
Define $\Dist(\mathbb Q^1 \rightarrow \mathbb E^1)$ as the set of all unordered pairs  of occurrences $\{{_{1{\bf a}}x},  {_{1{\bf a}}y}\}$ such that $x,y \in \con({\bf u})=\con({\bf v})$,  ${\bf a}$ is a block of ${\bf u}$ and one of the following holds: 
\begin{itemize}
\item {$({_{1{\bf a}}x}) <_{\bf a} ({_{1{\bf a}}y})$ but  $({_{1{\bf b}}y}) <_{\bf b} ({_{1{\bf b}}x})$;}
\item {$({_{1{\bf a}}y}) <_{\bf a} ({_{1{\bf a}}x})$ but  $({_{1{\bf b}}x}) <_{\bf b} ({_{1{\bf b}}y})$,}
\end{itemize}
where $\bf b$ is the block of $\bf v$ corresponding to $\bf a$. 
Notice that the dual to Observation~\ref{O: rb} implies that $\Dist(\mathbb Q^1 \rightarrow \mathbb E^1)({\bf u} \approx {\bf v})$ is empty if and only if  ${\bf u} \approx {\bf v}$ holds on $\mathbb E^1$.
 
Let $Z(x,y)$ denote the set of all words $\bf u$ with $\mul({\bf u})=\{x,y\}$ such that each block of $\bf u$ is either $x$ or $y$, and no two blocks containing the same letter are adjacent.
Denote:
\begin{itemize}
\item $\Phi = \{{\bf c} t x^2y^2 \approx {\bf c} t y x^2y\mid {\bf c} \in  Z  (x, y) \}$;
\item $\overline{\Phi} = \{x^2y^2 t {\bf d} \approx  xy^2x t {\bf d}\mid {\bf d} \in  Z  (x, y) \}$;
\item $\phi: xy t xy \approx xy t yxy$; 
\item $\overline{\phi}: xy t xy \approx xyx t xy$.
\end{itemize}

When ${_{1{\bf u}}x}$ and  ${_{1{\bf u}}y}$ are adjacent in $\bf u$ and $({_{1{\bf u}}x}) <_{\bf u} ({_{1{\bf u}}y})$, we write $({_{1{\bf u}}x}) \ll_{\bf u} ({_{1{\bf u}}y})$.

\begin{fact} 
\label{F: E23}
Let $\mathbb V$ be a variety such that $\mathbb Q^1 \vee \mathbb L_2^1 \vee \mathbb R_2^1 \subseteq \mathbb V \subseteq \mathbb E^1 \vee \overline{\mathbb E^1}$. 
Then every identity of $\mathbb V$  follows within $\mathbb E^1 \vee \overline{\mathbb E^1}$ from $(\Phi \cup \{\phi\}) \cap \Id(\mathbb V)$ and an identity of $\mathbb V \vee \mathbb E^1$.
\end{fact}

\begin{proof} 
Take an arbitrary identity ${\bf u} \approx {\bf v}$ of $\mathbb V$ which fails in $\mathbb E^1$.
Lemma~\ref{F: L2} and the dual to Observation~\ref{O: rb} imply that $\Dist(\mathbb Q^1 \rightarrow  \mathbb E^1)({\bf u} \approx {\bf v})$ is not empty. 
This means that  for some block $\bf a$ of $\bf u$ and $x,y \in \con({\bf u})$ we have  $({_{1{\bf a}}x}) <_{\bf u} ({_{1{\bf a}}y})$ but $({_{1{\bf b}}y}) <_{\bf v} ({_{1{\bf b}}x})$,  where $\bf b$ is the block of $\bf v$ corresponding to $\bf a$. 
Since $\mathbb E^1 \vee \overline{\mathbb E^1}\subset \mathbb E^1_3$ by Proposition~\ref{P: EE}, Corollary~\ref{C: adjacent-L} allows us to assume that  $\bf a$ begins with the sequence of first occurrences of all letters in $\con({\bf a})$. 
We can further assume that $({_{1{\bf a}}x}) \ll_{\bf u}({_{1{\bf a}}y})$. 
Indeed, if the initially chosen pair $\{{_{1{\bf a}}x}, {_{1{\bf a}}y}\} \in \Dist(\mathbb Q^1 \rightarrow  \mathbb E^1)({\bf u} \approx {\bf v})$ is not adjacent in $\bf a$, then either $\{{_{1{\bf a}}x}, {_{1{\bf a}}z}\} \in \Dist(\mathbb Q^1 \rightarrow  \mathbb E^1)({\bf u} \approx {\bf v})$ or $\{{_{1{\bf a}}z}, {_{1{\bf a}}y}\} \in \Dist(\mathbb Q^1 \rightarrow  \mathbb E^1)({\bf u} \approx {\bf v})$ for some $z \in \con({\bf a})$ with  $({_{1{\bf a}}x}) <_{\bf u} ({_{1{\bf a}}z}) <_{\bf u} ({_{1{\bf a}}y})$. 
Iterating this argument we find a pair which is adjacent in $\bf a$.

Let $t$ denote the simple letter to the left of $\bf a$ in $\mathbf u$ (we may assume without any loss that such a letter exists). 
We multiply both sides of ${\bf u} \approx {\bf v}$ by $y$ on the right, erase all multiple letters in ${\bf u} \approx {\bf v}$ other than $x$ and $y$, erase all simple letters on the right of block $\bf a$ and some simple letters on the left of $\bf a$ depending on which of the two cases takes place.

\smallskip

{\bf Case 1}: some block $\bf q$ on the left of $\bf a$ contains both $x$ and $y$.

In this case, the identity $xy{\bf u}y \approx xy{\bf v}y$ implies some identity which is equivalent modulo $\{xtx \approx xtx^2 \approx x^2tx,\, x^2y^2 \approx (xy)^2\}$ to $\phi$ (see Observation~\ref{O: same}).

Since  $\bf q$ contains both $x$ and $y$,  Corollary~\ref{C: uwu} implies that using $\{xtx \approx xtx^2 \approx x^2tx,\, x^2y^2 \approx (xy)^2\}$ the block $\bf q$ can be modified to some block which contains $xy$ as a subword. 
So, we may assume that $\bf q$ contains $xy$ as a subword. 
Then, since $({_{1{\bf a}}x}) \ll_{\mathbf u} ({_{1{\bf a}}y})$, we can apply $\phi$ to ${\bf u}$ so that the image of $x$ is a power of $x$ and the image of $y$ is a power of $y$.
As a result, we obtain ${\bf u}_1$ such that $({_{1{\bf a}_1}y}) \ll_{\mathbf u_1} ({_{1{\bf a}_1}x})$, where ${\bf a}_1$ is the block of ${\bf u}_1$ corresponding to $\bf a$.

\smallskip

{\bf Case 2}: no block on the left of $\bf a$ contains both $x$ and $y$.

In this case, the identity ${\bf u}y \approx {\bf v}y$ implies some identity which is equivalent modulo $\{xtx \approx xtx^2 \approx x^2tx,\, x^2y^2 \approx (xy)^2\}$ to an identity ${\bf r} \approx {\bf s} \in \Phi$.

Since $({_{1{\bf a}}x}) \ll_{\mathbf u} ({_{1{\bf a}}y})$, we can apply 
${\bf r} \approx {\bf s}$ to ${\bf u}$ so that the image of $x$ is a power of $x$ and the image of $y$ is a power of $y$.
As a result, we obtain ${\bf u}_1$ such that $({_{1{\bf a}_1}y}) \ll_{\mathbf u_1} ({_{1{\bf a}_1}x})$, where ${\bf a}_1$ is the block of ${\bf u}_1$ corresponding to $\bf a$.

\smallskip

Since in every case we have $|\Dist(\mathbb Q^1 \rightarrow \mathbb E^1)({\bf u}_1 \approx {\bf v})| < |\Dist(\mathbb Q^1 \rightarrow \mathbb E^1)({\bf u} \approx {\bf v})|$, Lemma~\ref{L: fb} implies that  every identity of $\mathbb V$ can be derived from $(\Phi \cup \{\phi\}) \cap \Id(\mathbb V)$ and  an identity of $\mathbb V \vee \mathbb E^1$.
\end{proof}

The following theorem generalizes Proposition~5.11(i) in \cite{Jackson-Lee-18} which says that the variety $\mathbb E^1$ is HFB.

\begin{theorem} 
\label{T: hfbEE}
Every subvariety of  $\mathbb E^1 \vee \overline{\mathbb E^1}$ is FB.
\end{theorem}

\begin{proof}  
Let $\mathbb V$ be a subvariety of  $\mathbb E^1 \vee \overline{\mathbb E^1}$.
If $\mathbb V$ does not contain $\mathbb Q^1$, then $\mathbb V$ contains neither  $\mathbb E^1\{\sigma_2\}$ nor $\overline{\mathbb A^1}$.
Then  $\mathbb V$ is FB by Lemma~\ref{L: not E}. 
Hence we may assume that $\mathbb V$ contains $\mathbb Q^1$.  
If either $\mathbb L_2^1$ or $\mathbb R_2^1$ is not contained in $\mathbb V$, then $\mathbb V$ is a subvariety of  $\mathbb E^1$ or $\overline{\mathbb E^1}$ by Fact~\ref{F: Ek} or the dual to it. 
Since $\mathbb E^1$ is HFB by Proposition~5.11(i) in~\cite{Jackson-Lee-18}, we may assume that $\mathbb V$ contains $\mathbb L_2^1 \vee \mathbb R_2^1$.

By Fact~\ref{F: E23}, every identity of $\mathbb V$  follows within $\mathbb E^1 \vee \overline{\mathbb E^1}$ from $(\Phi \cup \{\phi\}) \cap \Id(\mathbb V)$ and an identity of $\mathbb V \vee \mathbb E^1$.
The dual of Fact~\ref{F: E23} implies that every identity of  $\mathbb V \vee \mathbb E^1$ follows within $\mathbb E^1 \vee \overline{\mathbb E^1}$ from $(\overline{\Phi} \cup \{\overline{\phi}\}) \cap \Id(\mathbb V \vee \mathbb E^1) = (\overline{\Phi} \cup \{\overline{\phi}\}) \cap \Id(\mathbb V)$ and an identity of $\mathbb V \vee \mathbb E^1 \vee \overline{\mathbb E^1} = \mathbb E^1 \vee \overline{\mathbb E^1}$. 
Hence every identity of $\mathbb V$  follows within $\mathbb E^1 \vee \overline{\mathbb E^1}$ from $((\Phi \cup \{\phi\}) \cap \Id(\mathbb V)) \cup ((\overline{\Phi} \cup \{\overline{\phi}\}) \cap \Id(\mathbb V) )= (\Phi \cup \overline{\Phi} \cup \{ \phi, \overline{\phi}\}) \cap \Id(\mathbb V)$.
Clearly, any subset of $\Phi \cup \overline{\Phi} \cup \{ \phi, \overline{\phi}\}$ is equivalent to some its finite subset.
Therefore, $\mathbb V$ is FB.
\end{proof}

Since $\mathbb Q^1 \vee \mathbb L_2^1\vee \mathbb R_2^1$ satisfies $xy^2tx \approx xy^2x tx$ and $xt xy^2 \approx xt yxy$ by Facts~\ref{F: Q} and~\ref{F: L2}, the proof of Theorem~\ref{T: hfbEE} and Observation~\ref{O: same} give us the following.

\begin{ex} 
\label{E: QLR2} 
The variety
\[
\begin{aligned}
&\mathbb Q^1 \vee \mathbb L_2^1\vee \mathbb R_2^1=\\
&= \var \{ xtx \approx xtx^2 \approx x^2tx, (xy)^2 \approx x^2y^2, x^2y^2tx \approx xy^2x tx, xt x^2y^2 \approx xt yx^2y \}\\ 
&= \var \{ xtx \approx xtx^2 \approx x^2tx, (xy)^2 \approx x^2y^2, x^2y^2ty \approx xy^2x ty, yt x^2y^2 \approx yt yx^2y \}
\end{aligned}
\]
is HFB.\qed
\end{ex}

\section{New limit variety of monoids} 
\label{sec: new lim}

For a set of identities $\Sigma$ and $k>0$, we use $\Sigma_k$ to denote the set of all identities from $\Sigma$ which involve at most $k$ letters.

\begin{fact}[\mdseries{\cite[Fact~2.1]{Sapir-15N}; see also~\cite[Section~4.2]{Volkov-01}}]
\label{F: nfb} 
Let $\mathbb V = \var\,\Sigma$ be a semigroup variety and $\Sigma$ a set of identities.
Suppose that for infinitely many $n$, the variety $\mathbb V$  satisfies an identity ${\bf U}_n \approx {\bf V}_n$ in at least $n$ letters such that  ${\bf U}_n$ has some Property~\textup{(P${}_n$)} but  ${\bf V}_n$ does not. 
Suppose that for every word $\bf U$ such that $\mathbb V$ satisfies ${\bf U} \approx {\bf U}_n$ and ${\bf U}$  has Property~\textup{(P${}_n$)}, for every substitution $\Theta\colon\mathfrak A \rightarrow \mathfrak A^+$ and every identity ${\bf u} \approx {\bf v} \in \Sigma_k$ such that $\Theta({\bf u}) = {\bf U}$ and $k <n/2$, the word $\Theta({\bf v})$ also has property~\textup{(P${}_n$)}.   
Then $\mathbb V$ is NFB.\qed
\end{fact}

If  $\mathbf U=\Theta(\mathbf u)$ for some  substitution $\Theta\colon\mathfrak A \rightarrow \mathfrak A^+$ and $_{i{\bf U}}x$ is an occurrence of a letter $x$ in $\bf U$, then  $\Theta^{-1}_{\bf u}({_{i{\bf U}}x})$ denotes  an occurrence  ${_{j{\bf u}}z}$ of a letter $z$ in  $\bf u$ such that $\Theta({_{j{\bf U}}z})$ regarded as a subword of $\bf U$ contains $_{i{\bf U}}x$.

Recall that $\beta$ is the fully invariant congruence of $\mathbb E^1$.
Let $\overline{\beta}$ denote the congruence  dual to $\beta$.

\begin{sufcon} 
\label{SC2}
Let $\mathbb V$ be a monoid variety that  satisfies the identity
\begin{equation} 
\label{long identity2} 
{\bf U}_n = x y_1^2y_2^2\cdots y^2_{n-1} y_n^2x\approx x y_1^2 x y_2^2 \cdots y^2_{n-1} x y^2_n x = {\bf V}_n
\end{equation}
for any $n\ge1$.  
If the sets $[ab^2ta]_{\overline{\beta}}$,  $[at b^2 a]_{\beta}$ and $a a^+ b b^+$  are stable with respect to $\mathbb V$, then $\mathbb V$ is NFB.
\end{sufcon}

\begin{proof}
Consider the following property of a word ${\bf U}$ with $\con({\bf U})  = \{x, y_1, \dots,y_n\}$:
\begin{itemize}
\item[\textup{(P)}] there is no  $x$  in $\bf U$  between  the last occurrence of $y_1$ and the first occurrence  of $y_n$.
\end{itemize}
Notice that ${\mathbf U}_n$ satisfies Property~(P) but  ${\bf V}_n$ does not.

Let $\bf U$ be such that $\mathbb V$ satisfies ${\bf U}_n \approx {\bf U}$. 
Since $aa^+bb^+$ is stable with respect to $\mathbb V$,  we have:
\begin{equation}
\label{letters in u1}
({_{1{\bf U}}x})   <_{\bf U} ({_{\ell{\bf U}}y_1})  <_{\bf u} ({_{1{\bf U}}y_2}) <_{\bf U} ({_{\ell{\bf U}}y_2}) <_{\bf U} \dots <_{\bf U} ({_{\ell{\bf U}}y_{n-1}}) <_{\bf U} ({_{1{\bf U}}y_n}) <_{\bf U} ({_{\ell{\bf U}}x}).
\end{equation}

Let ${\bf u} \approx {\bf v}$ be an identity of $\mathbb V$ in less than $n/2$ letters and let
 $\Theta: \mathfrak A \rightarrow \mathfrak A^+$  be a substitution such that $\Theta({\bf u}) = {\bf U}$.
In view of \eqref{letters in u1}, the following holds:
\begin{itemize}
\item[\textup{($\ast$)}] for $t\in\con(\mathbf u)$, if $\Theta(t)$ contains both $y_i$ and $y_j$ for some $1\le i <j \le n$, then the letter $t$ is simple in $\bf u$.
\end{itemize}

Suppose that $\bf U$ has Property (P).  
Let us verify that ${\bf V} = \Theta({\bf v})$ also has Property~(P).
To obtain a contradiction, assume that there is  an occurrence of $x$ in $\bf V$ between ${_{\ell{\bf V}}y_1}$ and ${_{1{\bf V}}y_n}$.
By symmetry, we may assume without any loss that there is an occurrence ${_{k{\bf V}}x}$ of $x$ in $\bf V$ such that $({_{1{\bf V}}y_{n/2}}) <_{\bf V} ({_{k{\bf V}}x}) <_{\bf V} ({_{1{\bf V}}y_{n}})$.

Since $\bf u$ has less than $n/2$ letters, for some $t \in \con({\bf u})$ the word $\Theta(t)$ contains both $y_i$ and $y_j$ for some $1 \le i <j \le n/2$. 
In view of ($\ast$), the letter $t$ is simple in $\bf u$.
By Fact~\ref{F: abt}(i), $t$ is an isoterm for $\mathbb V$. 
Hence  the letter $t$ is simple in $\bf v$ as well.
Clearly, $\Theta^{-1}_{\bf v}({_{k{\bf V}}x}) = {_{p{\bf v}}z}$ is an occurrence of some letter $z$ in $\bf v$ such that $\Theta(z)$ contains $x$.
Since the empty word $1$ is an isoterm for $\mathbb V$ by Fact~\ref{F: abt}(i), the letter $z$ occurs in $\bf u$ as well.

In view of Fact~2.6 in \cite{Sapir-15N}, $\Theta^{-1}_{\bf u}({_{1{\bf U}}y_n}) = {_{1{\bf u}}y}$ and
 $\Theta^{-1}_{\bf v}({_{1{\bf V}}y_n}) = {_{1{\bf v}}y^\prime}$  for some $y,y^\prime \in \con({\bf u})=\con({\bf v})$.
If $y \ne y^\prime$ then $(_{1{\bf u}} y) <_{\bf u} {(_{1{\bf u}}y^\prime)}$ but $(_{1{\bf v}} y^\prime) <_{\bf v} {(_{1{\bf v}}y)}$.
This is impossible, because ${\bf u}(y,y^\prime)$ is $\beta$-term for $\mathbb V$ by Fact~\ref{F: abt}(i).
Thus $y=y^\prime$.

Further, in view of~\eqref{letters in u1} and the fact that $a a^+ b b^+$  is stable with respect to $\mathbb V$, we have:
\[
({_{1{\bf V}}x})   <_{\bf V} ({_{\ell{\bf V}}y_1})  <_{\bf u} ({_{1{\bf V}}y_2}) <_{\bf V} ({_{\ell{\bf V}}y_2}) <_{\bf V} \dots <_{\bf V} ({_{\ell{\bf V}}y_{n-1}}) <_{\bf V} ({_{1{\bf V}}y_n}) <_{\bf V} ({_{\ell{\bf V}}x}).
\]
Hence if $\Theta(z)$ is not a power of $x$, then $\Theta(z)$ contains either $xy_i$ for some $n/2\le i \le n$ or $y_jx$ for some $n/2 \le j <n$. 
This is impossible, because $\Theta(z)$ is a subword of $\bf U$ and $\bf U$ has Property (P) and \eqref{letters in u1}.
Therefore, $\Theta(z)$ is a power of $x$.
Then $z\ne y$ and $z\ne t$.
Since $({_{1{\bf V}}y_{n/2}}) <_{\bf V} ({_{k{\bf V}}x}) <_{\bf V} ({_{1{\bf V}}y_{n}})$, we have  $({_{{\bf v}}t}) <_{\bf v} ({_{p{\bf v}}z}) <_{\bf v} ({_{1{\bf v}}y})$.
In particular, $y\ne t$.

Since $\bf U$ has Property (P), no $z$ occurs between $t$ and ${_{1{\bf u}}y}$ in $\bf u$. 
Hence  ${\bf u}(z,y,t) \in  z^{\ast}t y \{y, z\}^{\ast}$.  
On the other hand, ${\bf v}(z, y, t) \in  z^{\ast}t z \{y, z\}^{\ast}$.   
This is impossible, because ${\bf u}(z,y,t)$ is $\beta$-term for $\mathbb V$ by  Fact~\ref{F: abt}(ii).
We conclude that $\bf V$ must also satisfy Property~(P). 
Therefore,  $\mathbb V$ is NFB by  Fact~\ref{F: nfb}.
\end{proof}

\begin{cor} 
\label{C: EE} 
Every monoid variety $\mathbb V$ that contains $\mathbb A^1_0 \vee {\mathbb E^1}\{\sigma_2\}  \vee \overline{{\mathbb E^1}\{\sigma_2\}}$ and is contained in $\mathbb A^1_0 \vee {\mathbb E^1} \vee \overline{{\mathbb E^1}}$ is NFB.
\end{cor}

\begin{proof} 
Facts~\ref{F: L2} and~\ref{F: E} and the dual to Fact~\ref{F: E} imply that  $\mathbb E^1 \vee \overline{{\mathbb E^1}}$ satisfies $\mathbf U_n\approx \mathbf V_n$ for each $n\ge 1$.
The variety $\mathbb A_0^1$ also satisfies $\mathbf U_n\approx \mathbf V_n$ by Proposition~4.2 in~\cite{Sapir-15}.

Since $\mathbb V$ contains $\overline{{\mathbb E^1}\{\sigma_2\}}$, ${\mathbb E^1}\{\sigma_2\}$
and $\mathbb A^1_0$, the sets $[ab^2ta]_{\overline{\beta}}$,  $[at b^2 a]_{\beta}$ and $a a^+ b b^+$  are stable with respect to $\mathbb V$ by Fact~\ref{F: E2} and the dual of  Fact~\ref{F: E2}(ii).
Hence $\mathbb V$ is NFB by the Sufficient Condition~\ref{SC2}.
\end{proof}

\begin{fact} 
\label{F: noA01} 
Let $\mathbb V$ be a variety satisfying $xtx \approx xtx^2 \approx x^2tx$. 
If $\mathbb V$ does not contain $\mathbb A_0^1$, then $\mathbb V$ satisfies $x^2y^2 \approx (xy)^2$.
\end{fact}

\begin{proof}  
Since  $\mathbb V$ does not contain $\mathbb A_0^1$, Fact~\ref{F: E2}(i) implies that $\mathbb V$ satisfies an identity $x^2y^2 \approx {\bf u}$ such that $\bf u$ contains $yx$ as a subword. Together with $xtx \approx xtx^2 \approx x^2tx$ this implies $x^2y^2 \approx (xy)^2$.
\end{proof}

\begin{theorem} 
\label{T: new limit}
The variety $\mathbb A^1_0 \vee \mathbb E^1\{\sigma_2\} \vee \overline{{\mathbb E^1}\{\sigma_2\}}$ is limit and is different from all the previously found limit varieties.
\end{theorem}

\begin{proof} The variety $\mathbb A^1_0 \vee {\mathbb E^1}\{\sigma_2\}  \vee \overline{{\mathbb E^1}\{\sigma_2\}}$ is NFB by Corollary~\ref{C: EE}.

Since ${\mathbb E^1}\{\sigma_2\}  \vee \overline{{\mathbb E^1}\{\sigma_2\}}$ is a subvariety of 
\[
\mathbb E^1 \vee \overline{\mathbb E^1} \stackrel{\text{Example~\ref{E: EE}}}{=} \var  \{xtx \approx xtx^2 \approx x^2tx, x^2y^2 \approx (xy)^2,  xs^2z^2x \approx xs^2xz^2x\},
 \]
it satisfies $xs^2z^2x \approx xs^2xz^2x$. 
Since $\mathbb A_0^1$ satisfies the identity $xszx \approx xsxzx$ (see~\cite[Proposition~3.2]{Edmunds-77} or \cite[Proposition~4.2]{Sapir-15}), the variety  $\mathbb A^1_0 \vee {\mathbb E^1}\{\sigma_2\}  \vee \overline{{\mathbb E^1}\{\sigma_2\}}$ satisfies $\{xtx \approx xtx^2 \approx x^2tx, xs^2z^2x \approx xs^2xz^2x\}$.

Let $\mathbb V$ be a proper subvariety of $\mathbb A^1_0 \vee {\mathbb E^1}\{\sigma_2\}  \vee \overline{{\mathbb E^1}\{\sigma_2\}}$.
If $\mathbb V$ does not contain either ${\mathbb E^1}\{\sigma_2\}$ or $\overline{{\mathbb E^1}\{\sigma_2\}}$, then $\mathbb V$ is FB by Lemma~\ref{L: not E}.
Hence we may assume that $\mathbb V$  does not contain $\mathbb A_0^1$. 
Then $\mathbb V$ satisfies $x^2y^2 \approx (xy)^2$ by Fact~\ref{F: noA01}.
We see that $\mathbb V$ is a subvariety of ${\mathbb E^1}  \vee \overline{{\mathbb E^1}}$.
Now Theorem~\ref{T: hfbEE} applies, yielding that the variety $\mathbb V$ is HFB.

Therefore, $\mathbb A^1_0 \vee {\mathbb E^1}\{\sigma_2\}  \vee \overline{{\mathbb E^1}\{\sigma_2\}}$ is a limit variety. 
It is different from $\mathbb A^1\vee\overline{\mathbb A^1}$, $\mathbb A^1 \vee {\mathbb E^1}\{\sigma_2\}$ and $\overline{\mathbb A^1} \vee \overline{{\mathbb E^1}\{\sigma_2\}}$ because it contains neither $\mathbb A^1$ nor $\overline{\mathbb A^1}$. 
No other limit variety mentioned Section~\ref{sec: known} contains $\mathbb A_0^1$ as a subvariety.
\end{proof}


\section{New Sorting Lemma}
\label{sec: new sort}

The following lemma is a combination of Sorting Lemma~2 in \cite{Gusev-Sapir-22} and Theorem~5.11 in~\cite{Gusev-Li-Zhang-25}.

\begin{sortlemma}\label{SL1}
Let $\mathbb V$ be a variety of aperiodic monoids.
Then either $\mathbb V$ is HFB or one of the following holds:
\begin{itemize}
\item[\textup{(i)}] $\mathbb V$ contains  one of the 11 varieties: $\mathbb A^1\vee\overline{\mathbb A^1}$, $\mathbb J$, $\overline{\mathbb J}$, $\mathbb J_1$, $\overline{\mathbb J_1}$,  $\mathbb J_2$, $\overline{\mathbb J_2}$, $\mathbb K$, $\overline{\mathbb K}$, $\mathbb M(\{xzxyty\})$, $\mathbb M(\{xyzxty,xzytxy\})$;
\item[\textup{(ii)}] $\mathbb V$ satisfies either  $\{xtx \approx xtx^2, xy^2tx \approx (xy)^2tx\}$  or dually,  $\{xtx \approx x^2tx, xty^2x  \approx xt(yx)^2\}$.\qed{\sloppy

}
\end{itemize}
\end{sortlemma}

Since articles \cite{Gusev-Sapir-22} and~\cite{Gusev-Li-Zhang-25} were published, we found three more limit varieties and are ready to prove the following.

\begin{sortlemma}
\label{SL2}
Let $\mathbb V$ be a variety of aperiodic monoids.
Then either $\mathbb V$ is HFB or one of the following holds:
\begin{itemize}
\item[\textup{(i)}]$\mathbb V$ contains  one of the 14 varieties: $\mathbb A^1\vee\overline{\mathbb A^1}$, $\mathbb J$, $\overline{\mathbb J}$, $\mathbb J_1$, $\overline{\mathbb J_1}$,  $\mathbb J_2$, $\overline{\mathbb J_2}$, $\mathbb K$, $\overline{\mathbb K}$, $\mathbb M(\{xzxyty\})$, $\mathbb M(\{xyzxty,xzytxy\})$, $\mathbb A^1 \vee {\mathbb E^1}\{\sigma_2\}$,  $\overline{\mathbb A^1} \vee \overline{{\mathbb E^1}\{\sigma_2\}}$, $\mathbb A^1_0 \vee {\mathbb E^1}\{\sigma_2\}  \vee \overline{{\mathbb E^1}\{\sigma_2\}}$;
\item[\textup{(ii)}] $\mathbb V$ satisfies $\{xtx \approx xtx^2 \approx x^2tx, (xy)^2 \approx x^2y^2\}$  and contains either $\mathbb M(x) \vee \mathbb L_3^1$ or $\mathbb M(x) \vee \mathbb R_3^1$.
\end{itemize}
\end{sortlemma}

\begin{proof}
Suppose that $\mathbb V$ is not HFB and does not contain any of the fourteen varieties. 
Then in view of the Sorting Lemma~\ref{SL1}, we may assume without any loss that $\mathbb V$ satisfies  $\{xtx \approx xtx^2, xy^2tx \approx (xy)^2tx\}$.
According to Theorem~4.3 in~\cite{Sapir-21}, $\mathbb V$ does not contain $\mathbb A^1$.
Then $\mathbb V$ contains $\overline{\mathbb E^1\{\sigma_2\}}$ by the dual to Lemma~\ref{L: not E}. 
Since $\mathbb V$ does not contain $\overline{\mathbb A^1}\vee\overline{\mathbb E^1\{\sigma_2\}}$, the variety $\mathbb V$ does not contain $\overline{\mathbb A^1}$.
Now  Lemma~\ref{L: not E} applies, yielding that $\mathbb V$ contains $\mathbb E^1\{\sigma_2\}$.
Hence $\mathbb E^1\{\sigma_2\} \vee \overline{\mathbb E^1\{\sigma_2\}}$ is a subvariety of $\mathbb V$.
Since $\mathbb V$ does not contain $\mathbb A^1_0 \vee {\mathbb E^1}\{\sigma_2\}  \vee \overline{{\mathbb E^1}\{\sigma_2\}}$, the variety $\mathbb V$ does not contain $\mathbb A_0^1$. 
Therefore, $\mathbb V$ satisfies $x^2y^2 \approx (xy)^2$ by Fact~\ref{F: noA01}.

Overall,  $\mathbb V$ satisfies  $\{xtx \approx xtx^2 \approx x^2tx, x^2y^2 \approx (xy)^2\}$ and contains  $\mathbb E^1\{\sigma_2\} \vee \overline{\mathbb E^1\{\sigma_2\}}$. 
If $\mathbb V$ contains neither $\mathbb L_3^1$ nor $\mathbb R_3^1$, then $\mathbb V$ is a subvariety of  $\mathbb E^1 \vee \overline{\mathbb E^1}$ by Corollary~\ref{C: EEk}. 
Since  $\mathbb E^1 \vee \overline{\mathbb E^1}$ is HFB by Theorem~\ref{T: hfbEE}, the variety $\mathbb V$ must contain either  $\mathbb L_3^1$ or $\mathbb R_3^1$.
Since  $\mathbb E^1\{\sigma_2\}$ contains $\mathbb M(x)$,  the variety $\mathbb V$ contains  either $\mathbb M(x) \vee \mathbb L_3^1$ or $\mathbb M(x) \vee \mathbb R_3^1$.\footnote{Using Fact~\ref{F: E2}(ii) and its dual, one can verify that $\mathbb M(x) \vee \mathbb R_3^1 = \mathbb E^1\{\sigma_2\} \vee \overline{\mathbb E^1\{\sigma_2\}} \vee \mathbb R_3^1$. }
\end{proof}

Sorting Lemma~\ref{SL2} leaves only two options for the variety  $\mathbb M(x) \vee \mathbb R_3^1$: to be either limit or HFB.
We are going to show that $\mathbb M(x) \vee \mathbb R_3^1$ is finitely based by
\[ 
\Delta = \{xtx \approx xtx^2 \approx x^2tx, x^2y^2 \approx (xy)^2, xyt x s y \approx xyx t x s y\}.
\]
The identity  $xyt x s y \approx xyx t x s y$ is contained in an infinite set of identities $\Psi$ which consists of all identities of the form $x {\bf Y}  {\bf B} \approx  x {\bf Y} x {\bf B}$ with $\con(x{\bf Y}) \subseteq \con({\bf B})$. 

\begin{lemma} 
\label{L: delta to sigma}  
The identities in $\Delta$ imply $\Psi$.
\end{lemma}

\begin{proof} 
Decompose 
\[
\Psi = \Psi_0 \cup \Psi_1 \cup \dots \cup \Psi_k\cup \cdots,
\]
where $\Psi_k$ is the set of all identities $\mathbf u=x {\bf Y} {\bf B} \approx  x {\bf Y} x {\bf B}=\mathbf v$ of $\Psi$ in which $\mathbf Y$ is of length $k$.
Verify by induction on $k$ that $\Delta$ implies $\Psi_k$.
If $k=0$, then the required claim is evident. 
Assume that for some $k \ge 0$ every identity in $\Psi_k$ follows from $\Delta$.
Now take an arbitrary identity $\mathbf u=x {\bf Y} {\bf B} \approx  x {\bf Y} x {\bf B}=\mathbf v\in\Psi_{k+1}$.
Then ${\bf Y} = y_1 y_2 \cdots  y_{k+1}$ (it is possible that for some $1 \le i < j \le k+1$ we have $y_i =y_j \in \con({\bf Y})$). 
Then, taking into account Observation~\ref{O: same}(iii), we have:
\begin{align*}
{\bf u} = x {\bf Y} {\bf B} =  x \cdot y_1 y_2 \cdots y_k \cdot y_{k+1} \cdot {\bf B} \stackrel{\Psi_k}{\approx}  x \cdot y_1 y_2 \cdots y_k \cdot x \cdot  y_{k+1} \cdot {\bf B}\stackrel{\Delta}\approx\\       
\stackrel{\Delta}{\approx}    x \cdot y_1 y_2 \cdots y_k \cdot x \cdot  y_{k+1} \cdot x \cdot {\bf B} \stackrel {\Psi_k}{\approx}   x \cdot y_1 y_2 \cdots  y_{k+1} \cdot x \cdot {\bf B} =  x {\bf Y}x {\bf B} = {\bf v}.
\end{align*}
Since the identity $\mathbf u\approx \mathbf v$ is arbitrary and $\Delta$ implies $\Psi_{k}$, we have $\Delta$ implies $\Psi_{k+1}$ as required.
\end{proof}

We say that an $\ell$-\textit{block} of $\bf u$ is a maximal subword of $\bf u$ which contains no last occurrences of letters. 
If  $\fin({\bf u}) = z_1 z_2 \cdots z_n$ for some $n>0$, then $\bf u$ has $n$  $\ell$-blocks ${\mathfrak a}_1, {\mathfrak a}_2, \dots ,  {\mathfrak a}_n$, some of which are empty:  
\[
{\bf u}  = {\mathfrak a}_1\cdot {_{\ell{\bf u}} z_1} \cdot {\mathfrak a}_2  \cdot{_{\ell{\bf u}}z_2} \cdots   {\mathfrak a}_n  \cdot{_{\ell{\bf u}}z_n}. 
\]
For each $1 \le i \le n$ we say that $z_i$ is the \textit{right divider} of ${\mathfrak a}_i$.
Notice that every $\ell$-block ${\mathfrak a}$ of a word $\bf u$ is contained in some block $\bf a$ of $\bf u$. 
Clearly, ${\mathfrak a}$ is the rightmost $\ell$-block of $\bf a$ if and only if the right divider $t$ of ${\mathfrak a}$ is simple in $\bf u$.
In this case, the letter $t$ is also  the \textit{right divider} of $\bf a$, that is, the simple letter adjacent to $\mathbf a$ on the right.
We say that $\bf u$ is $(\mathbb M(x) \vee \mathbb B^1)$-\textit{reduced} if for every  $\ell$-block ${\mathfrak a}$ of $\bf u$, every letter which occurs between ${\mathfrak a}$ and the right divider $t$ of $\bf a$, also occurs in ${\mathfrak a}$.

\begin{obs} 
\label{O: x=xx reduced} 
Every word $\bf u$  is equivalent modulo $\{xtx \approx xtx^2 \approx x^2tx, x^2y^2 \approx (xy)^2\}$ to a $(\mathbb M(x) \vee \mathbb B^1)$-reduced word.
\end{obs}

\begin{proof} 
Let $\bf a$ be a block of $\bf u$ which in turn is partitioned into $\ell$-blocks ${\mathfrak a}_1, {\mathfrak a}_2, \dots , {\mathfrak a}_k$ and their right dividers $t,z_2,\dots,z_k$ as follows: ${\bf a}={\mathfrak a}_k\cdot {_{\ell{\bf u}} z_k} \cdot {\mathfrak a}_{k-1}\cdot{_{\ell{\bf u}}z_{k-1}} \cdots   {\mathfrak a}_1  \cdot t$.
Suppose that for some $1\le m <k$,  for each $1 \le i \le m$ every letter which occurs between ${\mathfrak a}_i$ and the right divider $t$ of $\bf a$, also occurs in ${\mathfrak a}_i$.
Then for some words $\bf p$, $\bf s$ and some  $s \in \simp({\bf u})$, we have ${\bf u} = {\bf p} \cdot s \cdot {\bf a} \cdot t  \cdot {\bf s}$, and $\{xtx \approx xtx^2 \approx x^2tx, x^2y^2 \approx (xy)^2\}$ implies
\[
\begin{aligned}
{\bf u} &{}= {\bf p} s \cdot {\mathfrak a}_k\cdot {_{\ell{\bf u}} z_k} \cdots {\mathfrak a}_{m+1}\cdot{_{\ell{\bf u}}z_{m+1}}\cdot {\mathfrak a}_{m}\cdot{_{\ell{\bf u}}z_{m}}  \cdots{\mathfrak a}_1  \cdot t {\bf s}\\  
&{}\stackrel{\text{Lemma~\ref{L: xyxy}}}{\approx}  {\bf p} s \cdot {\mathfrak a}_k\cdot {_{\ell{\bf u}} z_k} \cdots {\mathfrak a}_{m+1}(z_{m+1}{\mathfrak a}_{m}z_{m}  \cdots{\mathfrak a}_1)\cdot{_{\ell{\bf u}}z_{m+1}}\cdot {\mathfrak a}_{m}\cdot{_{\ell{\bf u}}z_{m}}  \cdots{\mathfrak a}_1  \cdot t {\bf s} ={\bf v}.
\end{aligned}
\] 
In other words, using identities $\{xtx \approx xtx^2 \approx x^2tx, x^2y^2 \approx (xy)^2\}$ we can convert the $\ell$-block ${\mathfrak a}_{m+1}$ of ${\bf u}$ into ${\mathfrak b} = {\mathfrak a}_{m+1}(z_{m+1}{\mathfrak a}_{m}z_{m}  \cdots{\mathfrak a}_1)$ and result in $\mathbf v$. 
Notice that every letter which occurs between ${\mathfrak b}$ and $t$ in $\bf v$ also occurs in ${\mathfrak b}$.
Since we derived the identity ${\bf u} \approx {\bf v}$ from $\{xtx \approx xtx^2 \approx x^2tx, x^2y^2 \approx (xy)^2\}$, by induction, we can derive the identity  ${\bf u} \approx {\bf w}$ from $\{xtx \approx xtx^2 \approx x^2tx, x^2y^2 \approx (xy)^2\}$, where every $\ell$-block inside block $\bf b$ corresponding to $\bf a$ has the required property, while other blocks of $\bf w$ are the same as $\bf u$. 
We repeat this procedure to every block and eventually obtain a $(\mathbb M(x) \vee \mathbb B^1)$-reduced word.
\end{proof}

We say that $\bf u$ is $\Psi$-\textit{reduced} if $\bf u$  is $(\mathbb M(x) \vee \mathbb B^1)$-reduced and every letter appears at most once in every $\ell$-blocks of ${\bf u}$. 
We use $r_z({\bf u})$ to denote the maximal suffix of $\bf u$ which contains no occurrences of $z$.  
In view of Fact~\ref{F: R3}, an identity ${\bf u} \approx {\bf v}$ holds in $\mathbb R_3^1$ if and only if ${\bf u} \approx {\bf v}$ holds in $\mathbb L_2^1 \vee \mathbb R_2^1$ and for each  $z \in \con({\bf u})$ we have $\ini(r_z({\bf u}))= \ini (r_z({\bf v}))$.

\begin{lemma} 
\label{L: Psi reduced} 
\quad
\begin{itemize}
\item[\textup{(i)}] Every word $\bf w$  is equivalent modulo \[
\Psi \cup \{xtx \approx xtx^2 \approx x^2tx, x^2y^2 \approx (xy)^2\} 
\]
to a $\Psi$-reduced word.
\item[\textup{(ii)}] If ${\bf u}$ and ${\bf v}$ are $\Psi$-reduced words and $\mathbb M(x) \vee \mathbb R_3^1$ satisfies ${\bf u} \approx  {\bf v}$, then ${\bf u} =  {\bf v}$.
\end{itemize}
\end{lemma}

\begin{proof} 
(i) By Observation~\ref{O: x=xx reduced}, we can use $\{xtx \approx xtx^2 \approx x^2tx, x^2y^2 \approx (xy)^2\}$ to transform $\bf w$ into a $(\mathbb M(x) \vee \mathbb B^1)$-reduced word $\bf v$. 
Now we use identities in $\Psi$ to erase all non-first occurrences of letters in each $\ell$-block of $\bf v$.

\smallskip

(ii) Since ${\bf u} \approx {\bf v}$ holds in $\mathbb R_2^1$, Fact~\ref{F: L2}(ii) implies that  $\fin({\bf u}) = \fin({\bf v})$.   Let ${\mathfrak a}$  and ${\mathfrak b}$ be the corresponding $\ell$-blocks of ${\bf u}$ and ${\bf v}$.
Denote by $p$ the letter immediately before ${\mathfrak a}$ and ${\mathfrak b}$ on the left in $\mathbf u$ and $\mathbf v$ (we may assume that such a letter exists). 
Let $t \in \simp({\bf u})=\simp({\bf v})$ be the right divider of the corresponding blocks $\bf a$ of $\bf u$ and $\bf b$ of $\bf v$ which contain ${\mathfrak a}$  and ${\mathfrak b}$ (we may also assume that such a right divider exists).

To obtain a contradiction, assume that ${\mathfrak a}  \ne {\mathfrak b}$. 
Since ${\bf u} \approx {\bf v}$ holds in $\mathbb R_3^1$, we have $\ini (r_p({\bf u})) = \ini (r_p({\bf v}))$ by Fact~\ref{F: R3}.
Since each letter appears in ${\mathfrak a}$ only once, the word ${\mathfrak a}$ is a prefix of $\ini (r_p({\bf u}))$.
Similarly, the word ${\mathfrak b}$ is a prefix of $\ini (r_p({\bf v}))$.
Since ${\mathfrak a}  \ne {\mathfrak b}$, modulo duality, ${\mathfrak a}$ is a proper prefix of ${\mathfrak b}$. 
Then ${\mathfrak b} ={\mathfrak a} \cdot y_1y_2 \cdots y_m$ for some $m>0$.
Since each letter appears in ${\mathfrak b}$ only once, we have $\{y_1,y_2, \dots, y_m\} \cap \con({\mathfrak a}) = \emptyset$.
Since   $\ini (r_p({\bf u})) = \ini (r_p({\bf v}))$, the letters $y_1,y_2, \dots, y_m$ must be in $\bf u$ between ${\mathfrak a}$ and ${_{\bf u}t}$.
This contradicts the fact that $\mathbf u$ is $(\mathbb M(x) \vee \mathbb B^1)$-reduced.

Since $\fin({\bf u}) = \fin({\bf v})$ and all the corresponding $\ell$-blocks are equal to each other, we have ${\bf u} =  {\bf v}$.
\end{proof}

Recall that an identity ${\bf u} \approx {\bf v}$ holds in the 3-element cyclic monoid $M(x)$ if and only if $\simp({\bf u}) = \simp({\bf v})$ and $\mul({\bf u}) = \mul({\bf v})$.

\begin{theorem} 
\label{T: QR3} 
The variety 
\[
\begin{aligned}
\mathbb M(x) \vee \mathbb R_3^1 = \var \{xtx \approx xtx^2 \approx x^2tx, x^2y^2 \approx (xy)^2, xyt x s y \approx xyx t x s y\}\\
\stackrel{Observation~\ref{O: same}}{=}\var \{xtx \approx xtx^2 \approx x^2tx, x^2y^2 \approx (xy)^2, xyt y s x \approx xyx t y s x\}
\end{aligned}
\]
is HFB. 
\end{theorem} 
 
\begin{proof} 
First, notice that $\Delta$ holds in $\mathbb M(x)\vee\mathbb R_3^1$ by Fact~\ref{F: R3} and the description of identities of $\mathbb M(x)$.
Now let us verify that every identity ${\bf u} \approx {\bf v}$ of $\mathbb M(x) \vee \mathbb R_3^1$ follows from $\Delta$.
In view of Lemma~\ref{L: delta to sigma}, it is enough to derive  ${\bf u} \approx {\bf v}$ from $\Psi^\prime = \{xtx \approx xtx^2 \approx x^2tx, x^2y^2 \approx (xy)^2\} \cup \Psi$. 
By Lemma~\ref{L: Psi reduced}(i), $\Psi^\prime$ implies $\{ {\bf u}\approx {\bf u}_1, {\bf v}\approx {\bf v}_1\}$ for some $\Psi$-reduced words ${\bf u}_1$ and ${\bf v}_1$. 
Since $\mathbb M(x) \vee \mathbb R_3^1$ satisfies $\Psi^\prime$, and  $\mathbb M(x) \vee \mathbb R_3^1$ satisfies ${\bf u} \approx  {\bf v}$, we have  $\mathbb M(x) \vee \mathbb R_3^1$ satisfies ${\bf u}_1 \approx  {\bf v}_1$. 
Hence ${\bf u}_1 =  {\bf v}_1$ by Lemma~\ref{L: Psi reduced}(ii). 
Consequently, ${\bf u} \stackrel{\Psi^\prime}{\approx} {\bf u}_1 = {\bf v}_1 \stackrel{\Psi^\prime}{\approx} {\bf v}$.
Thus, we have proved that $\mathbb M(x) \vee \mathbb R_3^1=\var\,\Delta$.
Since $\mathbb M(x) \vee \mathbb R_3^1$ is FB, Sorting Lemma~\ref{SL2} implies that it is HFB.
\end{proof}

\begin{fact}  
\label{F: QRE1}
$\mathbb M(x) \vee \mathbb R_3^1 = \mathbb E^1 \vee \mathbb R_3^1$.
 \end{fact}

\begin{proof}  
Since by Figure~4 in \cite{Jackson-Lee-18},  the variety $\mathbb E^1$ contains $\mathbb M(x)$, we have 
$\mathbb M(x) \vee \mathbb R_3^1 \subseteq \mathbb E^1 \vee \mathbb R_3^1$.

Conversely, let  ${\bf u} \approx {\bf v}$ be an identity of $\mathbb M(x) \vee \mathbb R_3^1$.
Since ${\bf u} \approx {\bf v}$ holds in $\mathbb M(x)$, it is easy to see that $\simp({\bf u}) = \simp({\bf v})$. 
Since ${\bf u} \approx {\bf v}$ holds in $\mathbb R_2^1$, Fact~\ref{F: L2}(ii) implies that the sequence of simple letters is the same in $\bf u$ and $\bf v$.  
Let ${\mathbf a}$  and ${\mathbf b}$ be the corresponding blocks of ${\bf u}$ and ${\bf v}$.
Let $s, t \in \simp({\bf u})=  \simp({\bf v})$ be the left and right dividers of $\bf a$ in $\bf u$ and $\bf b$ in $\bf v$.
In view of Fact~\ref{F: R3}, we have $\ini (r_s({\bf u})) = \ini (r_s({\bf v}))$.
Since $\ini({\bf a}) \cdot t$ is a prefix of $\ini (r_s({\bf u}))$ and $\ini({\bf b}) \cdot t$ is a prefix of $\ini (r_s({\bf v})) = \ini (r_s({\bf u}))$, we have $\ini({\bf a}) = \ini({\bf b})$. 
Therefore, ${\bf u} \approx {\bf v}$ holds in  $\mathbb E^1$ by Facts~\ref{F: L2}(i) and~\ref{F: E}. 
Thus, $\mathbb M(x) \vee \mathbb R_3^1 = \mathbb E^1 \vee \mathbb R_3^1$.
\end{proof}

In view of Sorting Lemma~\ref{SL2} and Theorem~\ref{T: QR3}, if there exists any other limit variety of aperiodic monoids, then it is contained in
\[
\mathbb M(x) \vee \mathbb B^1 \stackrel{\text{Fact}~\ref{F: QRE1}}{=}  \mathbb Q^1 \vee \mathbb B^1 \stackrel{\text{Lemma}~\ref{L: xyxy}}{=} \var\{xtx \approx xtx^2 \approx x^2tx,\, x^2y^2 \approx (xy)^2\}
\]
and properly contains either $\mathbb M(x) \vee \mathbb L_3^1$ or $\mathbb M(x)  \vee \mathbb R_3^1$.

\begin{question} 
\label{q}
Is every monoid satisfying the identities $xtx \approx xtx^2 \approx x^2tx$ and $x^2y^2 \approx (xy)^2$ finitely based?
\end{question}


Institute of Natural Sciences and Mathematics, Ural Federal University, Lenina 51, 620000 Ekaterinburg, Russia

\textit{E-mail address}: \texttt{sergey.gusb@gmail.com}

\medskip

Department of Mathematics, Ben-Gurion University of the Negev, P.O.B. 653, 8410501 Beersheba, Israel

\textit{E-mail address}: \texttt{olga.sapir@gmail.com}

\end{document}